\documentclass[a4paper]{amsart}

\usepackage{graphicx, psfrag, layout, appendix, subfigure}
\usepackage{fancyhdr}

\usepackage{amssymb, amsthm, amsmath, amscd, amsfonts, mathabx, mathrsfs, latexsym}

\usepackage{enumitem, multirow}
\usepackage{color}
\usepackage{url}

\usepackage{setspace}
\usepackage{changepage}

\usepackage{leftidx}

\usepackage{bbold}

\usepackage[active]{srcltx}

\usepackage{calc}

\usepackage{pb-diagram,pb-xy}

\usepackage{times}

\usepackage{verbatim}
\usepackage[all,cmtip,arrow,matrix,curve,tips,frame]{xy}
\usepackage[
	setpagesize=false,
	colorlinks=true,
	linkcolor=blue,
	pdfencoding=auto,
	psdextra,
]{hyperref}
\usepackage{cleveref}
\xyoption{matrix}

\usepackage{xcolor}
\hypersetup{
    colorlinks,
    linkcolor={red!50!black},
    citecolor={blue!50!black},
    urlcolor={blue!80!black}
}

\usepackage{ifthen}
\usepackage{color}
\usepackage{amsxtra}
\usepackage{amstext}
\usepackage{stmaryrd}
\usepackage{mathrsfs}
\usepackage{epsfig}
\usepackage{pifont}
\usepackage{charter}
\usepackage{avant}
\usepackage{courier}
\usepackage[T1]{fontenc}
\usepackage{textcomp}
\usepackage{bbm}
\usepackage{tikz-cd}

\setlist[enumerate,1]{label=(\arabic*), ref=(\arabic*)}
\setlist[enumerate,3]{label=(\roman*), ref=(\roman*)}

\numberwithin{equation}{section}

\theoremstyle{plain}
\newtheorem{thm}{Theorem}[section]
\newtheorem*{thm*}{Theorem}

\newtheorem{lem}[thm]{Lemma}

\newtheorem*{lem*}{Lemma}

\newtheorem{prop}[thm]{Proposition}

\newtheorem*{prop*}{Proposition}

\newtheorem*{cor*}{Corollary}

\newtheorem*{claim*}{Claim}

\newtheorem{conj}[thm]{Conjecture}
\newtheorem*{conj*}{Conjecture}

\theoremstyle{definition}
\newtheorem{defn}[thm]{Definition}
\newtheorem*{defn*}{Definition}

\newtheorem*{ques*}{Question}

\newtheorem{exa}[thm]{Example}

\newtheorem*{exa*}{Example}

\theoremstyle{remark}
\newtheorem{rmk}[thm]{Remark}

\newtheorem*{rmk*}{Remark}

\newtheorem*{notation}{Notation}

\numberwithin{figure}{section}
\numberwithin{table}{section}
\numberwithin{equation}{section}

\def \CC {\mathbb{C}}

\def \FF {\mathbb{F}}

\def \PP {\mathbb{P}}
\def \QQ {\mathbb{Q}}

\def \ZZ {\mathbb{Z}}

\def \Ecal {\mathcal{E}}
\def \Fcal {\mathcal{F}}

\def \Hcal {\mathcal{H}}
\def \Ical {\mathcal{I}}

\def \Kcal {\mathcal{K}}

\def \Ocal {\mathcal{O}}

\def \Qcal {\mathcal{Q}}

\def \Scal {\mathcal{S}}
\def \Tcal {\mathcal{T}}
\def \Ucal {\mathcal{U}}
\def \Vcal {\mathcal{V}}

\def \Xcal {\mathcal{X}}
\def \Ycal {\mathcal{Y}}
\def \Zcal {\mathcal{Z}}

\def \Lfr {\mathfrak{L}}

\def \hbar {\bar{h}}

\DeclareMathOperator{\lcm}{lcm}

\DeclareMathOperator{\Sym}{Sym}

\DeclareMathOperator{\Proj}{Proj}

\DeclareMathOperator{\sign}{sign}

\DeclareMathOperator{\codim}{codim}

\DeclareMathOperator{\Aut}{Aut}

\DeclareMathOperator{\GL}{GL}

\DeclareMathOperator{\PGL}{PGL}

\DeclareMathOperator{\Hom}{Hom}

\DeclareMathOperator{\Gr}{Gr}

\title{Algebraic Hyperbolicity of Surfaces in Fano threefolds with Picard number one}
\author{Haesong Seo}
\address{Department of Mathematical Sciences, KAIST, 291 Daehak-ro, Yuseong-gu, Deajeon, 34141, Republic of Korea
}
\email{hss21@kaist.ac.kr}


\begin{document}

\maketitle

\begin{abstract}
    In this paper, we study the algebraic hyperbolicity of very general surfaces in general Fano threefolds with Picard number one.
    We completely classify the algebraically hyperbolicity of those surfaces, except for surfaces in weighted hypersurfaces.
\end{abstract}

%---------------------------------------------------------------------

\section{Introduction}\label{sec:intro}

We will work over $\CC$.
A projective variety $X$ is \textit{algebraically hyperbolic} if there exist an ample divisor $H$ and a real number $\varepsilon > 0$ such that for any integral curve $C \subset X$, we have
\begin{equation}
    2g(\overline{C})-2 \geq \varepsilon \deg_{C}(H)
\end{equation}
where $\overline{C}$ is the normalization of $C$ and $g(\overline{C})$ is the genus of $\overline{C}$.

The concept of algebraic hyperbolicity is proposed by Demailly \cite{Demailly1997} as an algebraic analogue of Kobayashi hyperbolicity.
A complex manifold is called \textit{Kobayashi hyperbolic} if the Kobayashi pseudometric is nondegenerate.
Brody \cite{Brody1978} proved that a compact complex manifold $X$ is Kobayashi hyperbolic if and only if it is \textit{Brody hyperbolic}, i.e., there is no nonconstant entire map $f:\CC \rightarrow X$.
Demailly \cite{Demailly1997} proved that a (Kobayashi) hyperbolic complex projective manifold is algebraically hyperbolic, and conjectured the converse.

Although it is in general very difficult to determine whether a given variety is algebraically hyperbolic, there has been significant progress toward classifying the algebraic hyperbolicity of hypersurfaces of (almost) homogeneous varieties.
Recently, Hasse and Ilten \cite{HI2021} introduced the study of the algebraic hyperbolicity of surfaces in toric threefolds.
Based on a technique developed by Ein \cite{Ein1988, Ein1991}, Pacienza \cite{Pacienza2003} and Voisin \cite{Voisin1996, Voisin1998}, Coskun and Riedl \cite{CR2019, CR2023} completely classified the algebraic hyperbolicity of surfaces in several toric threefolds, including $\PP^3$, $\PP^1 \times \PP^1 \times \PP^1$, $\PP^2 \times \PP^1$, $\FF_e \times \PP^1$ and the blowup of $\PP^3$ at one point, where $\FF_e = \PP(\Ocal_{\PP^1} \oplus \Ocal_{\PP^1}(e))$.
Following \cite{HI2021}, Robins \cite{Robins2023} almost classified the algebraic hyperbolicity of surfaces in smooth projective toric threefolds with Picard number two and three.
For higher dimensions, Yeong \cite{Yeong2022} completely classified the algebraic hyperbolicity of general hypersurfaces in the projective space $\PP^m$ and the product $\PP^m \times \PP^n$, except for sextic threefolds or hypersurfaces of certain bidegrees in $\PP^3 \times \PP^1$.
More generally, Mioranci \cite{Mioranci2023} provided an almost optimal degree bound for hypersurfaces in homogeneous varieties to be algebraically hyperbolic.
Also, Moraga and Yeong \cite{MY2024} showed that toric projective manifolds satisfy the following hyperbolicity conjecture.

\begin{conj} [{\cite[Conjecture~1.1]{MY2024}}] \label{conj:Hyperbolicity-conjecture}
    Let $L$ be an ample Cartier divisor on a smooth projective manifold $X$ of dimension $n \geq 2$.
    Then a very general element of the linear system $|K_X + (3n+1)L|$ is algebraically hyperbolic.
\end{conj}

One remarkable point is that the idea of Coskun and Riedl \cite{CR2019} can easily be generalized to complete intersections, or even the zero loci of regular sections of homogeneous vector bundles on (almost) homogeneous varieties.
Hence it would be reasonable to study the algebraic hyperbolicity of surfaces in Fano threefolds, considering Mukai's description of general Fano threefolds as the zero loci of general sections of homogeneous vector bundles on the product of (weighted) Grassmannians (cf. \cite{Mukai1989}, see also \cite{CCGK2016} and \cite{DBFT2021}).
Here, by general Fano manifolds we mean that they correspond to general points of their moduli.
Note that we cannot argue about surfaces in special Fano threefolds using these tools, e.g., the Mukai-Umemura threefold $X$ of genus $12$ with $\Aut(X) \simeq \PGL_2$.

In \cite{CR2019}, they proved that very general surfaces of degree $d \geq 5$ in $\PP^3$ are algebraically hyperbolic.
Generalizing their work, it is natural to ask whether very general surfaces of degree greater than the Fano index in Fano threefolds with Picard number one are algebraically hyperbolic.

\begin{thm} \label{thm:main-theorem}
    Let $X$ be a general Fano threefold with Picard number one.
    Let $\Ocal_X(1)$ be a Picard generator and $r$ the Fano index, i.e., $-K_X = \Ocal_X(r)$ in the Picard group of $X$.
    For $a \geq 1$, let $S \in |\Ocal_X(a)|$ be a very general surface.
    \begin{itemize}
        \item If $X$ can be written as a weighted hypersurface, i.e., one of \hyperlink{Fano1-1}{\textbf{1-1}}, \hyperlink{Fano1-11}{\textbf{1-11}} and \hyperlink{Fano1-12}{\textbf{1-12}}, then $S$ is algebraically hyperbolic when $a \geq r+2$.
        \item For all the other cases, $S$ is algebraically hyperbolic if and only if $a \geq r+1$.
    \end{itemize}
\end{thm}

The technique yields more refined classifications for weighted complete intersections.
It is also noteworthy that we obtain a finer bound for the hyperbolicity conjecture~\ref{conj:Hyperbolicity-conjecture} on general Fano threefolds.

As featured in various papers including \cite{CR2019} and \cite{Mioranci2023}, the main obstacle lies in dealing with surfaces of degree $r+1$.
The scroll argument \ref{lem:surface-scroll} is effective for complete intersections, while it does not provide an optimal bound for weighted hypersurfaces.
Further, this argument fails to apply to certain Fano threefolds \hyperlink{Fano1-8}{\textbf{1-8}}, \hyperlink{Fano1-9}{\textbf{1-9}} and \hyperlink{Fano1-10}{\textbf{1-10}}.

To outline the proof, for a given curve $f:C \rightarrow S$, we aim to capture the positivity of the normal sheaf $N_{f/S}$, as it yields the degree bound for the curve.
For this, we use the generalized notion of section-dominating family, which was originally introduced by \cite{CR2023}.
To deal with weighted hypersurfaces \hyperlink{Fano1-1}{\textbf{1-1}}, \hyperlink{Fano1-11}{\textbf{1-11}} and \hyperlink{Fano1-12}{\textbf{1-12}}, we determine a section-dominating line bundle on the weighted projective space in Theorem~\ref{thm:Almost-section-dominating-on-WPS}.
Then we obtain a generic surjection $M_{\Ecal}|_C \rightarrow N_{f/S}$ from the Lazarsfeld-Mukai bundle of a certain homogeneous sheaf $\Ecal$ on the weighted Grassmannian.
In the case of \hyperlink{Fano1-8}{\textbf{1-8}}, \hyperlink{Fano1-9}{\textbf{1-9}} and \hyperlink{Fano1-10}{\textbf{1-10}}, the vector bundle $\Ecal$ need not be decomposed into line bundles.
Nevertheless, we prove that the $\QQ$-vector bundle $M_{\Ecal} \langle \frac{H}{2} \rangle$ is nef on $C$, where $H$ is a hyperplane section, which provides a stronger bound.
All the other cases are almost straightforward from the argument of Coskun and Riedl \cite{CR2019}.

The paper is organized as follows.
In Section~\ref{sec:preliminaries}, we first introduce the general setup of Coskun and Riedl \cite{CR2023}, and formulate Bott's theorem for Grassmannians.
In Section~\ref{sec:technical-lemmas}, we present two technical lemmas regarding the weighted projective spaces and vanishing results of certain homogeneous vector bundles on the Grassmannians.
In Section~\ref{sec:main-results}, we prove our main theorem \ref{thm:main-theorem} by dealing with each deformation class of Fano threefolds separately.

\begin{notation}
    On the Grassmannian $\Gr(k,n)$, denote by $\Ocal(1)$ the Pl\"ucker line bundle.
    Also, denote by $\Ucal$ and $\Qcal$ the universal and the quotient bundle, respectively, which fit into the tautological sequence
    \begin{equation}
    \begin{tikzcd}
        0 \arrow{r} & \Ucal \arrow{r} & \Ocal^{\oplus n} \arrow{r} & \Qcal \arrow{r} & 0.
    \end{tikzcd}
    \end{equation}
    Note that the ranks of $\Ucal$ and $\Qcal$ are $k$ and $n-k$, respectively.
    The zero locus of a regular section of a vector bundle $\Ecal$ on an ambient variety $A$ is denoted by $\Lfr(\Ecal) \subset A$.
    For all the other notations, we follow \cite{Hartshorne1977}.
\end{notation}

\vspace{8pt} \noindent \textbf{Acknowledgements.} The author is supported by the Institute for Basic Science (IBS-R032-D1).
The author would like to express our gratitude to Professor Yongnam Lee for his suggestions on research topics and valuable comments.
The author would also like to thank for Professor Izzet Coskun for very useful comments.

%---------------------------------------------------------------------

\section{Preliminaries}\label{sec:preliminaries}

\subsection{Setup.} \label{subsec:setup}
We follow the setting of \cite{CR2023} in a generalized fashion.
Let $A$ be a normal projective variety of dimension $n$ equipped with a group action by an algebraic group $G$.
Suppose that the $G$-action has a smooth open orbit $A_0 \subset A$ with $\codim(A \setminus A_0) \geq 2$.
Let $\Fcal$ be a globally generated $G$-invariant vector bundle of rank $r < n-2$ on $A$.
Let $L$ be a $G$-invariant reflexive coherent sheaf of generic rank one that is locally free and globally generated on $A_0$.
Let $\Ecal = \Fcal \oplus L$.

We require some additional assumptions for the setup. \\
\hypertarget{Assump1}{\textbf{Assumption 1}}. General sections of $\Fcal$ and $\Ecal$ are regular in $A$ and $A_0$, respectively, and the zero locus $S$ of a general section of $\Ecal$ is a smooth subvariety contained in $A_0$. \\
\hypertarget{Assump2}{\textbf{Assumption 2}}. The canonical line bundle $K_S$ is ample on $S$. \\
\hypertarget{Assump3}{\textbf{Assumption 3}}. For the zero locus $X$ of a general section of $\Fcal$, we have a surjection $H^0(A,L) \rightarrow H^0(X,L)$.

The Assumption \hyperlink{Assump3}{\textbf{3}} says that $S$ is the zero locus of a general section of $L|_X$ on a general $X$.
In fact, the Assumption \hyperlink{Assump3}{\textbf{3}} can be obtained by the following vanishing condition of intermediate cohomology groups:
\begin{equation} \label{eqn:restriction-map-is-surjective}
    H^{i} \left( A, \bigwedge^{i} \Fcal^{\vee} \otimes L \right) = 0
\end{equation}
for $1 \leq i \leq r$.
Indeed, from the Assumption \hyperlink{Assump1}{\textbf{1}} one can consider the Koszul resolution
\[
\begin{tikzcd}
    K^{\bullet}(\Fcal^{\vee}) \arrow[r] & \Ocal_X \arrow[r] & 0,
\end{tikzcd}
\]
where
\[
    K^{i}(\Fcal^{\vee}) = \bigwedge^{-i-1} \Fcal^{\vee}
\]
for $i \leq -1$.
We are done from the spectral sequence and the condition~\ref{eqn:restriction-map-is-surjective}.

Let $\Scal_1 \rightarrow \Vcal \subset H^0(A,\Ecal)$ be the universal family of the vanishing loci of regular sections of $\Ecal$.
Then it factors through the universal family $\Xcal_1 \rightarrow \Vcal$ of the vanishing loci of regular sections of $\Fcal$.
Suppose that there exists a map $f:C \rightarrow S$ from a curve of genus $g$.
Let $\Hcal \rightarrow \Vcal$ be the relative Hilbert scheme with universal curves $\Ycal_1 \rightarrow \Scal_1$.
Since $S$ has a finite automorphism group from the Assumption \hyperlink{Assump2}{\textbf{2}}, there exists a $G$-invariant subvariety $\Ucal \subset \Hcal$ such that $\Ucal \rightarrow \Vcal$ is \'etale by a standard argument (cf. \cite[Section~1]{CR2004}).
Resolving the singularity of generic fibers of $\Ycal_1|_{\Ucal}$ and restricting $\Ucal$ further (if necessary), we get a smooth family $\Ycal \rightarrow \Ucal$ of curves of genus $g$.
By taking the pullbacks $\Scal$ and $\Xcal$ of $\Scal_1$ and $\Xcal_1$ to $\Ucal$, respectively, denote by
\begin{equation}
\begin{tikzcd}
    \Ycal \arrow{r}{h} & \Scal \arrow{r}{i} & \Xcal \arrow{r}{\pi_2} \arrow{d}{\pi_1} & A \\
    & & \Ucal &
\end{tikzcd}
\end{equation}
the resulting morphisms.
Note that $h$ is generically injective and $i$ is immersive.
Also, the composition $\pi_2 \circ i \circ h$ dominates $A$ because $\Ycal$ is invariant under the action of $G$.
For $t \in \Ucal$, let $Y_t$, $S_t$ and $X_t$ be the fiber of $\Ycal$, $\Scal$ and $\Xcal$ over $t$, respectively, and write $h_t = h|_{Y_t}:Y_t \rightarrow S_t$ and $i_t = i|_{S_t}:S_t \rightarrow X_t$.

Define the \textit{vertical tangent sheaf} of $\Ycal$ over $A$ by $T_{\Ycal/A} = \ker(T_{\Ycal} \rightarrow h^{\ast}i^{\ast}\pi_2^{\ast}T_A)$, and similarly define $T_{\Scal/A}$ and $T_{\Xcal/A}$.
For a coherent sheaf $\Ecal$ on a projective variety $A$, we define the \textit{Lazarsfeld-Mukai sheaf} $M_{\Ecal}$ of $\Ecal$ by
\begin{equation}
    M_{\Ecal} = \ker(H^0(A,\Ecal) \otimes \Ocal_A \xrightarrow{\mathrm{ev}} \Ecal),
\end{equation}
where $\mathrm{ev}$ denotes the evaluation map.

\begin{prop}[cf. {\cite[Proposition~2.1]{CR2023}}] \label{prop:basic-diagrams}
    Under the above setting, assume further that the image of $\pi_2 \circ i:\Scal \rightarrow A$ is in $A_0$.
    For general $t \in \Ucal$,
    \begin{enumerate}
        \item $T_{\Xcal/A} \simeq \pi_2^{\ast}M_{\Fcal}$;
        \item $T_{\Scal/A} \simeq i^{\ast}\pi_2^{\ast}M_{\Ecal}$;
        \item $N_{h_t/S_t} \simeq N_{h/\Scal}|_{Y_t}$;
        \item There exists a map $T_{\Scal/A} \rightarrow N_{h/\Scal}$ that is surjective on $A_0$.
    \end{enumerate}
\end{prop}

\begin{proof}
    The normal sheaf $N_{\Xcal/\Ucal \times A}$ is isomorphic to $\pi_{2}^{\ast}\Fcal$ by \cite[Example~6.3.4]{Fulton1984}.
    Also, we have $N_{\Scal/\Ucal \times A} \simeq i^{\ast}\pi_2^{\ast}\Ecal$ by the assumption.
    The same proof would be applied except for (2).
    For (2), the isomorphism comes from the following diagram.
     \[
      \begin{tikzcd}
        & & 0 & & \\
        & i^{\ast}\pi_2^{\ast}\Ecal \arrow[r, "\simeq"] & N_{\Scal/\Ucal \times A} \arrow{u} & & \\
        0 \arrow{r} & H^0(A,\Ecal) \otimes \Ocal \arrow{r} \arrow{u} & i^{\ast}\pi_2^{\ast}T_A \oplus H^0(A,\Ecal) \otimes \Ocal \arrow{r} \arrow{u} & i^{\ast}\pi_2^{\ast}T_A \arrow{r} & 0 \\
        0 \arrow{r} & T_{\Scal/A} \arrow{r} \arrow{u} & T_{\Scal} \arrow{r} \arrow{u} & i^{\ast}\pi_2^{\ast}T_A \arrow{r} \arrow[u, equal] & 0 \\
        & 0 \arrow{u} & 0 \arrow{u} & &
      \end{tikzcd}
     \]
    Here, the second column is the definition of the normal sheaf.
\end{proof}

\subsection{Section-dominating family.} \label{subsec:section-dominating}
Using a generic surjection $M_{\Ecal}|_{Y_t} \rightarrow N_{h_t/S_t}$, one can derive a lower bound for the degree of $N_{h_t/S_t}$. To deal with the weighted projective spaces, we generalize the notion of a section-dominating collection of line bundles.

\begin{defn} [cf. {\cite[Definition~2.3]{CR2023}}]
    Let $A$ be a normal projective variety, and $A_0 \subset A$ a big open subset, i.e., $A \setminus A_0$ has codimension at least two in $A$.
    Let $\Ecal$ be a reflexive coherent sheaf on $A$, and $L_1, \dots, L_m$ a collection of nontrivial reflexive coherent sheaves on $A$.
    The family $L_1, \dots, L_m$ is called \textit{section-dominating} for $\Ecal$ on $A_0$ if
    \begin{enumerate}
        \item $\Ecal$ is locally free on $A_0$;
        \item $L_1, \dots, L_m$ are locally free and globally generated on $A_0$; and
        \item the multiplication map
            \begin{equation}
            \begin{tikzcd}
                \bigoplus_{i=1}^m H^0(L_i \otimes \Ical_p) \otimes H^0(\Ecal [\otimes] L_i^{\vee}) \ar[r] &
                H^0(\Ecal \otimes \Ical_p)
            \end{tikzcd}
            \end{equation}
            is surjective for any $p \in A_0$, where $\Ical_p$ is the ideal sheaf of $p$ and
            \[
                \Ecal [\otimes] L_i^{\vee} = (\Ecal \otimes L_i^{\vee})^{\vee \vee}.
            \]
    \end{enumerate}
    If $A_0 = A$, then we just call it section-dominating for $\Ecal$.
\end{defn}

\begin{exa} [cf. {\cite[Example~2.4]{CR2023}}, {\cite[Example~2.5]{Mioranci2023}}] \label{exa:section-dominating-on-Grassmannian}
    On the projective space $\PP^{n}$, the line bundle $\Ocal_{\PP^n}(1)$ is section-dominating for $\Ocal_{\PP^n}(a)$ if $a \geq 1$.
    Also, on the Grassmannian $\Gr(k,n)$, the Pl\"ucker line bundle $\Ocal(1)$ is section-dominating for $\Ocal(a)$ if $a \geq 1$ because the Pl\"ucker embedding is projectively normal.
\end{exa}

\begin{lem}[cf. {\cite[Proposition~2.6]{CR2023}}] \label{lem:section-dominating-Mukai-bundle}
    Let $A$ be a normal projective variety, and $A_0 \subset A$ a big open subset.
    Let $\Ecal$ be a reflexive coherent sheaf on $A$.
    If a family $L_1, \dots, L_m$ of reflexive coherent sheaves on $A$ forms a section-dominating family for $\Ecal$ on $A_0$, there is a map $\oplus_{i=1}^u M_{L_i}^{\oplus s} \rightarrow M_{\Ecal}$ that is surjective on $A_0$ for some integer $s$.
\end{lem}

The same proof as in \cite{CR2023} works.
If $S$ is a surface (i.e., $r = n-3$), as the normal sheaf $N_{h_t/S_t}$ has rank one, there is some $L_i$ together with a generically surjective map $M_{L_i}|_{Y_t} \rightarrow N_{h_t/S_t}$.

\begin{prop}[cf. {\cite[Proposition~2.7]{CR2023}}] \label{prop:Mukai-bundle-degree}
    Let $A$ be a normal projective variety, and $A_0 \subset A$ a big open subset.
    Let $L$ be a reflexive coherent sheaf on $A$ that is locally free and globally generated on $A_0$.
    Let $f:C \rightarrow A$ be a map from a curve into $A$ whose image lies in $A_0$.
    If there is a generic surjection from $M_L|_C$ to a line bundle $N$ on $C$, we have $\deg_C N \geq -\deg_C L$.
\end{prop}

\begin{proof}
    From the sequence
    \[
    \begin{tikzcd}
        0 \ar[r] & M_L \ar[r] & H^0(A,L) \otimes \Ocal_{A} \ar[r] & L,
    \end{tikzcd}
    \]
    consider its dual
    \[
    \begin{tikzcd}
        0 \ar[r] & L^{\vee} \ar[r] & H^0(A, L)^{\vee} \otimes \Ocal_{A} \ar[r] & M_L^{\vee}.
    \end{tikzcd}
    \]
    The map $H^0(A, L)^{\vee} \otimes \Ocal_{A} \rightarrow M_L^{\vee}$ is surjective on $A_0$, so there is a map
    \[
    \begin{tikzcd}
        \bigwedge^{r-1} H^0(A,L)^{\vee} \otimes \Ocal_{A} \arrow[r] & \bigwedge^{[r-1]} M_L^{\vee} = (\bigwedge^{r-1} M_L^{\vee})^{\vee\vee}
    \end{tikzcd}
    \]
    that is surjective on $A_0$, where $r$ is the generic rank of $M_L$.
    Hence the sheaf $M_L [\otimes] \det L$ is globally generated on $A_0$, where $\det L = \bigwedge^{[\mathrm{rk}(L)]} L$.
    Since the image of the map
    \[
    \begin{tikzcd}
        (M_L [\otimes] \det L)|_C = M_L|_C \otimes \det L|_C \arrow[r] &
        N (\det L|_C)
    \end{tikzcd}
    \]
    is also globally generated, its degree is nonnegative.
    Since the cokernel is torsion, we have $\deg_C N(\det L|_C) \geq 0$, i.e., $\deg_C N \geq -\deg_C L$.
\end{proof}

\subsection{Surface scrolls.} \label{subsec:surface-scroll}
Recall that a surface $S$ in $\PP^n$ is said to be a \textit{scroll} if through any point of $S$, there passes a line contained in $S$.

\begin{lem} [cf. {\cite[Lemma~2.12]{CR2023}}] \label{lem:surface-scroll}
    Under the setting of Section~\ref{subsec:setup}, let $H$ be a very ample line bundle on $A$.
    Consider the embedding $A \hookrightarrow \PP^n$ induced by $H$.
    Then a rank one quotient $M_{H}|_{Y_t} \rightarrow Q$ on $Y_t$ induces a scroll over $Y_t$ in $\PP^n$ whose $\Ocal_{\PP^n}(1)$-degree is equal to $\deg_{Y_t} Q + (Y_t.H)$.
\end{lem}

\subsection{Bott's theorem for Grassmannians.} \label{subsec:bott-theorem}
We will use Bott's theorem \ref{thm:Bott-theorem-for-Grassmannians} to analyze whether a certain homogeneous vector bundle is section-dominating for another homogeneous vector bundle on the Grassmannian.
Let us introduce some notions in representation theory.
A \textit{partition} $\alpha$ of a natural number $n$ is a sequence $\alpha = (\alpha_1, \dots, \alpha_n) \in \ZZ^{n}$ such that $\alpha_1 \geq \cdots \geq \alpha_n \geq 0$ and $\alpha_1 + \cdots + \alpha_n = n$.
We often omit 0 in the expression of a partition, e.g.,
\[
    (\alpha_1, \dots, \alpha_k, 0, \dots, 0) = (\alpha_1, \dots, \alpha_k).
\]
The \textit{conjugate partition} $\alpha'$ of a partition $\alpha$ is given by
 \[
  \alpha_i' = |\{ j: \alpha_j \geq i \}|
 \]
for integers $1 \leq i \leq n$.

Let $E$ be a vector space of dimension $n$ over $\CC$.
Denote by $\bigwedge^r E$ the $r$-th exterior power of $E$ for $0 \leq r \leq n$.
For $0 \leq p,q \leq n$ with $p+q \leq n$, there is the natural multiplication map
 \[
  \mu: \bigwedge^p E \otimes \bigwedge^q E \rightarrow \bigwedge^{p+q} E
 \]
given by
 \[
  \mu(u_1 \wedge \cdots \wedge u_p \otimes v_1 \wedge \cdots \wedge v_q) = u_1 \wedge \cdots \wedge u_p \wedge v_1 \wedge \cdots \wedge v_q
 \]
for $u_1, \dots, u_p, v_1, \dots, v_q \in E$.
Also, there is the natural comultiplication map
 \[
  \Delta: \bigwedge^{p+q} E \rightarrow \bigwedge^p E \otimes \bigwedge^q E
 \]
given by
 \[
  \Delta(u_1 \wedge \cdots \wedge u_{p+q}) = \sum_{\sigma} (-1)^{\sign \sigma} u_{\sigma(1)} \wedge \cdots \wedge u_{\sigma(p)} \otimes u_{\sigma(p+1)} \wedge \cdots \wedge u_{\sigma(p+q)}
 \]
for $u_1, \dots, u_{p+q} \in E$, where the sum is taken over permutations $\sigma \in \Sigma_{p+q}$ on $p+q$ letters satisfying
 \[
  \sigma(1) < \cdots < \sigma(p), \qquad \sigma(p+1) < \cdots < \sigma(p+q).
 \]

\begin{defn} [cf. {\cite[Section~2.1]{Weyman2003}}]
    For a partition $\alpha = (\alpha_1, \dots, \alpha_s)$, the \textit{Schur module} $L_{\alpha}E$ corresponding to the partition $\alpha$ is defined as
    \begin{equation}
        L_{\alpha}E = \bigwedge^{\alpha_1} E \otimes \cdots \otimes \bigwedge^{\alpha_s} E / R(\alpha,E),
    \end{equation}
    where the relation $R(\alpha,E)$ is the sum of subspaces
     \[
      \bigwedge^{\alpha_1} E \otimes \cdots \otimes \bigwedge^{\alpha_{i-1}}E \otimes R_{i,i+1}(E) \otimes \bigwedge^{\alpha_{i+2}} E \otimes \cdots \otimes \bigwedge^{\alpha_s} E
     \]
    for $1 \leq i \leq s-1$, and where
    \[
        R_{i,i+1}(E) \subset \bigwedge^{\alpha_i} E \otimes \bigwedge^{\alpha_{i+1}} E
    \]
    is the subspace generated by the images of the maps
    \begin{equation}
    \begin{aligned}
        & \theta(\alpha, i, u, v;E): \bigwedge^u E \otimes \bigwedge^{\alpha_i-u+\alpha_{i+1}-v} E \otimes \bigwedge^v E \\
        &\qquad \xrightarrow{1 \otimes \Delta \otimes 1} \bigwedge^u E \otimes \bigwedge^{\alpha_i-u} E \otimes \bigwedge^{\alpha_{i+1}-v} E \otimes \bigwedge^v E \\
        &\qquad \xrightarrow{\mu_{12} \otimes \mu_{34}} \bigwedge^{\alpha_i} E \otimes \bigwedge^{\alpha_{i+1}} E
    \end{aligned}
    \end{equation}
    with $u+v < \alpha_{i+1}$.
\end{defn}

\begin{exa} [cf. {\cite[Example~2.1.3]{Weyman2003}}]
    \hfill
    \begin{enumerate} [label=(\alph*)]
        \item If $\alpha = (r)$, $L_{\alpha}E = \bigwedge^r E$.
        \item If $\alpha = (1^r)$, $L_{\alpha}E = \mathrm{Sym}^r E$.
        \item If $\alpha = (p, 1^{q-1})$, then the Schur module $L_{\alpha}E$ is the cokernel of the map
         \[
          \bigwedge^{p+1}E \otimes \mathrm{Sym}^{q-2}E \xrightarrow{\Delta \otimes 1} \bigwedge^p E \otimes E \otimes \mathrm{Sym}^{q-2}E \xrightarrow{1 \otimes \mu} \bigwedge^p E \otimes \mathrm{Sym}^{q-1}E,
         \]
        where $\mu:E \otimes \mathrm{Sym}^{q-2}E \rightarrow \mathrm{Sym}^{q-1}E$ is the natural multiplication map.
    \end{enumerate}
\end{exa}

For a weight $\alpha = (\alpha_1, \dots, \alpha_n)$ with $\alpha_1 \geq \cdots \geq \alpha_n$, define the \textit{Weyl module} on $E$ by
\begin{equation}
    K_{\alpha} E = L_{\bar{\alpha}'} E \otimes \left( \bigwedge^n E \right)^{\otimes \alpha_n},
\end{equation}
where $\bar{\alpha} = (\alpha_1-\alpha_n, \dots, \alpha_{n-1}-\alpha_n, 0)$.
Consider the action of the permutation group $\Sigma_n$ on $\ZZ^n$ given by
\begin{equation}
    (i,i+1) \cdot \alpha = (\alpha_1, \dots, \alpha_{i-1}, \alpha_{i+1}-1, \alpha_i+1, \alpha_{i+2}, \dots, \alpha_n)
\end{equation}
for each $1 \leq i < n$, where $\alpha = (\alpha_1, \dots, \alpha_n) \in \ZZ^n$.
This action can be described as
 \[
  \sigma \cdot \alpha = \sigma(\alpha+\rho) - \rho,
 \]
where
 \[
  \sigma((\alpha_1, \dots, \alpha_n)) = (\alpha_{\sigma(1)}, \dots, \alpha_{\sigma(n)})
 \]
and $\rho = (n-1, \dots, 1, 0)$.
Now we can state Bott's theorem for Grassmannians.

\begin{thm} [Bott-Borel-Weil~theorem, {\cite[Corollary~4.1.9]{Weyman2003}}] \label{thm:Bott-theorem-for-Grassmannians}
    Let $E$ be a $\CC$-vector space of dimension $n$.
    Consider a weight $\alpha = (\alpha_1, \dots, \alpha_n) \in \ZZ^n$ satisfying $\alpha_i \geq \alpha_{i+1}$ for $i \neq r$.
    Let $\beta = (\alpha_1, \dots, \alpha_r)$ and $\gamma = (\alpha_{r+1}, \dots, \alpha_n)$ be two weights, and define
    \begin{equation}
        \Vcal(\alpha) = K_{\beta} \Ucal^{\vee} \otimes K_{\gamma} \Qcal^{\vee}
    \end{equation}
    on the Grassmannian $\Gr(r,E)$.
    Then one of two mutually exclusive possibilities occurs:
    \begin{enumerate}
        \item If $\sigma \cdot \alpha = \alpha$ for some $1 \neq \sigma \in \Sigma_n$, the cohomology groups $H^{i}(\Gr(k,E), \Vcal(\alpha))$ are zero for $i \geq 0$.
        \item If there exists a unique $\sigma \in \Sigma_n$ such that $\sigma \cdot \alpha = \beta$ is nonincreasing, all the cohomology groups $H^{i}(\Gr(k,E), \Vcal(\alpha))$ are zero for $i \neq \ell(\sigma)$, and
         \[
          H^{\ell(\sigma)}(\Gr(k,E), \Vcal(\alpha)) = K_{\beta} E^{\vee}
         \]
        where $\ell(\sigma)$ denotes the minimal length $\ell$ of a reduced expression
        \[
            \sigma = \sigma_1 \dots \sigma_{\ell}
        \]
        with $\sigma_i = (j, j+1)$ for some $1 \leq j < n$.
    \end{enumerate}
\end{thm}

%---------------------------------------------------------------------

\section{Technical lemmas}\label{sec:technical-lemmas}

\subsection{Weighted projective space.} \label{subsec:weighted-projective-spaces}

There are certain Fano threefolds, namely \hyperlink{Fano1-1}{\textbf{1-1}}, \hyperlink{Fano1-11}{\textbf{1-11}} and \hyperlink{Fano1-12}{\textbf{1-12}}, which can be represented as weighted hypersurfaces.
Hence we need to find a pair of integers $d,a \geq 1$ such that $\Ocal_A(d)$ is section-dominating for $\Ocal_A(a)$ on a big open subset of a weighted projective space $A$.

Let $k \geq 2$ and $2 \leq a_1 \leq \cdots \leq a_{\ell}$ be integers.
Let $V$ be a $k$-dimensional vector space, and let $R_1, \dots, R_{\ell}$ be one-dimensional vector spaces.
Define a $\CC^{\times}$-action on $V \oplus R_1 \oplus \cdots \oplus R_{\ell}$ by
\[
    c \cdot (x, y_1, \dots, y_{\ell}) = (cx, c^{a_1}y_1, \dots, c^{a_{\ell}}y_{\ell}),
\]
for $c \in \CC^{\times}$, $x \in V$, and $y_i \in R_i$ where $1 \leq i \leq \ell$.
The weighted projective space $A = \PP(1^k,a_1,\dots,a_{\ell})$ can be identified with the quotient
\[
    (V \oplus R_1 \oplus \cdots \oplus R_{\ell} \setminus \{ 0 \}) / \CC^{\times}.
\]
Consider the action of a semi-direct product
\[
    G = \left( \prod_{i=1}^{\ell} \Hom(\Sym^{a_{i}}V, R_i) \right) \rtimes \GL(V)
\]
on $A$ given by
\[
    L \cdot [x:y_1:\cdots:y_{\ell}] = [Lx:y_1:\cdots:y_{\ell}]
\]
and
\[
    \lambda_i \cdot [x:y_1:\cdots:y_{\ell}] = [x:y_1:\cdots: y_i+\lambda_i(x) :\cdots:y_{\ell}],
\]
where $L \in \GL(V)$, $x \in V$, $y_j \in R_j$ and $\lambda_i \in \Hom(\Sym^{a_{i}}V, R_i)$.
One can see that the orbit $A_0 = G \cdot p$ of the point $p = [1:0^{k-1}:1^{\ell}]$ is a big open subset of $A$ and the coherent sheaves $\Ocal_A(a)$ are invariant under the action of $G$ (cf. \cite[Proposition~2.6]{HK2015}).

\begin{thm} \label{thm:Almost-section-dominating-on-WPS}
    Under the above setting, assume that $\ell = 1$ or $d = \lcm(a_1, \dots, a_{\ell}) \neq a_{\ell}$.
    Then for any integer $a \geq \max \{d,\ell d - \sum_{i=1}^{\ell} a_i + 1\}$, the line bundle $\Ocal_{A}(d)$ is section-dominating for $\Ocal_{A}(a)$ on $A_0$.
\end{thm}

\begin{proof}
    The statement is direct when $\ell = 0$ (i.e., $A$ is the projective space $\PP^{k-1}$), so assume that $\ell \geq 1$.
    By the invariance, we only need to see that the multiplication map
    \[
        \Phi: H^0(\Ocal(d) \otimes \Ical_{p}) \otimes H^{0}(\Ocal(a-d)) \rightarrow H^0(\Ocal(a) \otimes \Ical_{p})
    \]
    is surjective.
    Note that we have $\Ocal(a) [\otimes] \Ocal(b) \simeq \Ocal(a+b)$, while $\Ocal(a) \otimes \Ocal(b) \not\simeq \Ocal(a+b)$ in general (cf. \cite[Pathologies~1.5]{Dolgachev1982}).

    Let $s_1, \dots, s_k$ be linear coordinates on $V$, and $t_i$ a linear coordinate on $R_i$.
    In order to verify that the map $\Phi$ is surjective, give a graded lexicographic order on $\CC[s_1, \dots, s_k, t_1, \dots, t_{\ell}]$ such that $s_1 < \dots < s_k$ and
    \[
        f(s_1, \dots, s_k, t_1, \dots, t_{i-1}) < t_i
    \]
    for any monomial $f$ of weighted degree $a_i$.
    For a nonzero polynomial $f \in H^0(\Ocal(a) \otimes \Ical_{p})$, we will use the induction on the leading term $\mathrm{lt}(f)$ of $f$ to show that $f$ is contained in the image of the map $\Phi$.
    
    First, we assume that $\mathrm{lt}(f)$ is divisible by none of $s_i$ for $2 \leq i \leq k$.
    If $\mathrm{lt}(f)$ is not divisible by $t_i$'s, then $f$ itself is a multiple of $s_1^{a}$, which does not vanish at $p$.
    Hence $\mathrm{lt}(f)$ is divisible by some $t_i$.

    We claim that $\mathrm{lt}(f)$ is divisible by a monomial of degree $d$.
    Once we prove this assertion, one can write $\mathrm{lt}(f) = gq$ for some monomials $g$ and $q$ of degree $a-d$ and $d$, respectively, with $q$ monic.
    If $q \neq s_1^{d}$, then $q - s_1^d$ is a nonzero element of $H^0(\Ocal(d) \otimes \Ical_p)$.
    So one can replace $f$ by $f - g(q-s_1^d)$ to reduce the leading term.
    If $q = s_1^d$, then we may replace $g$ and $q$ by $g' = s_1^{a_i}(g/t_i)$ and $q' = s_1^{d-a_i}t_i$, respectively, for some dividend $t_i$ of $\mathrm{lt}(f)$.
    Then one can replace $f$ by $f - g(q-s_1^d)$ as before.

    To prove the claim, let $q$ be a monomial of the largest degree among the monic monomials of degree at most $d$ that divide $\mathrm{lt}(f)$ but that are not divisible by $s_1$.
    We are done if $\deg q = d$, so assume that $\deg q < d$.
    Write $\mathrm{lt}(f) = g q$ for some monomial $g$.
    If $g$ is divisible by $s_1^{d-\deg q}$, then $\mathrm{lt}(f)$ is divisible by $s_1^{d-\deg q}q$.
    This is obviously satisfied when $\ell = 1$, so assume that $\ell \geq 2$.
    If $s_1$ does not divide $\mathrm{lt}(f)$, the pigeonhole principle yields $t_i^{d/a_i} \mid \mathrm{lt}(f)$ for some $i$.
    This contradicts $\deg q < d$, so $s_1$ divides $\mathrm{lt}(f)$.

    Now assume that $s_1^{d-\deg q}$ does not divide $\mathrm{lt}(f)$.
    Let $\alpha$ be the multiplicity of $s_1$ in $\mathrm{lt}(f)$.
    Note that $g$ is divisible by some $t_i$.
    For such an $i$, we have $\deg t_i q > d$ by maximality, i.e., $d-\deg q < a_i$.
    If $t_i^{d/a_i-1} \nmid \mathrm{lt}(f)$, then $\alpha$ is at least
     \[
      \left( \ell \cdot d - \sum_{j=1}^{\ell} a_j + 1 \right) - \sum_{j \neq i} \left( \frac{d}{a_j}-1 \right) a_j - \left( \frac{d}{a_i}-2 \right) a_i = a_i + 1,
     \]
    i.e., $d-\deg q > \alpha \geq a_i+1$, a contradiction.
    A similar argument says that $t_j^{d/a_j-1} \mid \mathrm{lt}(f)$ if $a_j \geq a_i$.
    In particular, $t_{\ell}^{d/a_{\ell}-1}$ divides $\mathrm{lt}(f)$ and $\alpha < a_{\ell}$.
    
    Let $\alpha_1$ and $\alpha_2$ be the multiplicities of $t_1$ and $t_2$ in $\mathrm{lt}(f)$, respectively.
    As above, we infer that
     \begin{align*}
      \alpha &\geq \left( \ell \cdot d - \sum_{j=1}^{\ell} a_j + 1 \right) - \alpha_1 a_1 - \alpha_2 a_2 - \sum_{j \geq 3} \left( \frac{d}{a_j}-1 \right) a_j \\
        &= 2d - \left( \alpha_1 + 1 \right) a_1 - \left( \alpha_2 + 1 \right) a_2 + 1.
     \end{align*}
    We claim that $t_{j}^{d/a_j - 1}$ divides $\mathrm{lt}(f)$ for $j \geq 2$.
    This is satisfied when either $t_1$ or $t_2$ divides $g$, so suppose that both do not divide $g$.
    Since $t_1^{\alpha_1}t_2^{\alpha_2} \mid q$, we have
     \[
      (2d - \left( \alpha_1 + 1 \right) a_1 - \left( \alpha_2 + 1 \right) a_2+1) + \alpha_1 a_1 + \alpha_2 a_2
        = 2d-a_1-a_2 + 1 < d,
     \]
    which is impossible because $a_1,a_2 \leq \frac{d}{2}$.
    
    Now observe that $\alpha_1$ is at least $\frac{d}{6a_1}$.
    Otherwise, from the inequality
     \[
      \alpha \geq d - (\alpha_1 + 1)a_1 + 1 > \frac{5d}{6} - a_1,
     \]
    we get $a_1 > \frac{d}{3}$ as $\alpha < a_{\ell} \leq \frac{d}{2}$.
    This is not the case.
    Consider the monomial $q' = t_1^{\beta}t_{\ell}^{d/a_{\ell}-1}$ of degree at most $d$ with $1 \leq \beta \leq \alpha_1$ maximal.
    Then $q'$ divides $\mathrm{lt}(f)$, so $\deg q' \leq \deg q$ and $\alpha \leq d-\frac{d}{2}-\frac{d}{6} = \frac{d}{3}$.
    Applying the above procedure repeatedly, we get $\alpha_1 \geq \frac{d}{2a_1}$.
    Since the monomial $t_1^{\alpha_1}t_{\ell}^{d/a_{\ell}-1}$ has degree at least $d$, $\alpha$ is strictly less than $a_1$.
    We infer that
    \[
        d - (\alpha_1+1)a_1 + 1 \leq \alpha < a_1,
    \]
    or
    \[
        \alpha_1 > \frac{d}{a_1} + \frac{1}{a_1} -2.
    \]
    Hence one can conclude that $t_i^{d/a_i-1} \mid \mathrm{lt}(f)$ for all $1 \leq i \leq \ell$.

    We are now able to derive a contradiction.
    If $k = \gcd(a_i, a_j) > 1$ for some $1 \leq i < j \leq \ell$, then the monomial
     \[
      q' = t_i^{\frac{1}{k} \cdot \frac{d}{a_i}} t_j^{\frac{k-1}{k} \cdot \frac{d}{a_j}}
     \]
    divides $\mathrm{lt}(f)$ and $\deg(q') = d$, which is a contradiction.
    Suppose that $a_1, \dots, a_{\ell}$ are pairwise coprime.
    Let us introduce the following classical result:

\begin{lem} \label{lem:coin-problem}
    Let $x, y \in \ZZ_{\geq 2}$ be coprime integers. Let
     \[
      a \equiv -x^{-1}\, (\text{mod } y), \qquad b \equiv -y^{-1}\, (\text{mod } x)
     \]
    for some $0 < a < y$ and $0 < b < x$.
    Then $ax + by = xy - 1$.
\end{lem}

\begin{proof}
    We give a proof for the sake of completeness.
    As $x$ and $y$ are coprime, there exist integers $c,d \in \ZZ$ such that $cx+dy = xy-1$.
    We may assume that $0 < c < y$.
    Then $0 < dy = (y-c)x-1 < xy$, so $0 < d < x$.
    Clearly $a = c$ and $b = d$.
\end{proof}

    Choose $0 < \alpha < a_{\ell}$ and $0 < \beta < \frac{d}{a_{\ell}} = a_1 \cdots a_{\ell-1}$ such that
     \[
      \alpha \cdot \frac{d}{a_{\ell}} + \beta \cdot a_{\ell} = d - 1
     \]
    and consider the monomial
     \[
      q' = s_1 t_1^{\alpha \cdot \frac{d}{a_1 a_{\ell}}} t_{\ell}^{\beta}.
     \]
    It divides $\mathrm{lt}(f)$ and $\deg(q') = d$, whence a contradiction.

    Finally, if $\mathrm{lt}(f)$ is divisible by $s_i$ for some $2 \leq i \leq k$, a similar argument shows that $\mathrm{lt}(f)$ is divisible by a (monic) monomial $q$ of degree $d$ divisible by $s_i$.
    Then we can replace $f$ by $f - \mathrm{lt}(f)$.
\end{proof}

\begin{rmk}
    When $\ell = 2$ and $a_1$, $a_2$ are coprime, the above result is optimal.
    Using the notation above, $t_1^{a_2-1}t_2^{a_1-1} - s_1^{2a_1a_2-a_1-a_2}$ is a section of $\Ocal_{X}(2a_1a_2-a_1-a_2)$ vanishing at $p$, which is not contained in the image of $\Phi$.
\end{rmk}

\subsection{Cohomology of homogeneous vector bundles.} \label{subsec:cohomology-of-homogeneous}

In this section, we prove the vanishing of cohomology groups of certain homogeneous vector bundles on the Grassmannian.
The following theorem will be useful when dealing with Fano threefolds \hyperlink{Fano1-8}{\textbf{1-8}}, \hyperlink{Fano1-9}{\textbf{1-9}} and \hyperlink{Fano1-10}{\textbf{1-10}}.

\begin{thm} \label{thm:Cohomology-vanishing-of-homogeneous}
    Let $n$ and $k$ be integers with $n \geq 2$, $1 \leq k \leq \frac{n}{2}$.
    For a given sequence of integers
    \[
        n-k \geq i_1 \geq \cdots \geq i_{k-1} \geq i_k = 0
    \]
    and integers $i \geq 0$ and $d > 0$, if $i_{\ell-1} - i_{\ell} < n-k$ for any $\ell$, we have
    \[
        H^i(\mathrm{Gr}(k,n), K_{(i_1, \dots, i_k)} \Ucal^{\vee}(-d)) \neq 0 
    \]
    if and only if $i = k(n-k)$ and $d \geq n+i_1$.
    If
    \[
        i_1 = \dots = i_{\ell} = n-k, \quad i_{\ell+1} = \dots = i_{n-k} = 0
    \]
    for some $0 \leq \ell < k$,
    \[
        H^{i}(\mathrm{Gr}(k,n), K_{((n-k)^{\ell})}\Ucal^{\vee}(-d)) \neq 0
    \]
    if and only if
    \[
        \begin{cases} i = (n-k)(k-\ell), \, d = n-\ell, \\ i = k(n-k), \, d \geq n, \end{cases}
    \]
    and in the former case,
    \[
        H^{(n-k)(k-\ell)}(\mathrm{Gr}(k,n), K_{((n-k)^{\ell})}\Ucal^{\vee}(-(n-\ell))) \simeq \CC.
    \]
    Similarly, for a given sequence of integers
    \[
        k \geq j_1 \geq \cdots \geq j_{n-k-1} \geq j_{n-k} = 0
    \]
    and integers $i, d \geq 0$, if $j_{\ell-1} - j_{\ell} < k$ for any $\ell$, then
    \[
        H^i(\mathrm{Gr}(k,n), K_{(j_1, \dots, j_{n-k})} \Qcal^{\vee}(-d)) \neq 0
    \]
    if and only if $i = k(n-k)$ and $d \geq n$.
    If
    \[
        j_1 = \dots = j_{\ell} = k, \quad j_{\ell+1} = \dots = j_{n-k} = 0
    \]
    for some $0 \leq \ell < n-k$,
    \[
        H^{i}(\mathrm{Gr}(k,n), K_{(k^{\ell})}\Qcal^{\vee}(-d)) \neq 0
    \]
    if and only if
    \[
        \begin{cases} i = k\ell, \, d = \ell, \\ i = k(n-k), \, d \geq n, \end{cases}
    \]
    and in the former case,
    \[
        H^{k\ell}(\mathrm{Gr}(k,n), K_{(k^{\ell})}\Qcal^{\vee}(-\ell)) \simeq \CC.
    \]
\end{thm}

\begin{proof}
    Let $E = \CC^n$ be an $n$-dimensional vector space and regard $\Gr(k,n)$ as $\Gr(k,E)$.
    Note that
    \[
        \Vcal((i_1, \dots, i_k, d^{n-k})) = K_{(i_1, \dots, i_k)} \Ucal^{\vee} \otimes \left( \det \Qcal^{\vee} \right)^{\otimes d} = K_{(i_1, \dots, i_k)} \Ucal^{\vee}(-d).
    \]
    Let $\alpha = (i_1, \dots, i_k, d^{n-k})$ for simplicity.
    Then $\alpha$ is not a partition as $d \geq 1$.
    If $d = 1$, then $\alpha_{k+1} = \alpha_k+1$, so $(k,k+1) \cdot \alpha = \alpha$.
    Hence by Bott's theorem~\ref{thm:Bott-theorem-for-Grassmannians}, all the cohomology groups vanish.
    
    For $d \geq 2$, in order to find $\sigma \in \Sigma_n$ with $\sigma \cdot \alpha$ nonincreasing, we first get rid of zero at the $k$-th entry:
    \[
        (k,k+1) \cdot \alpha = (i_1, \dots, i_{k-1}, d-1, 1, d^{n-k-1}).
    \]
    If $d = 2$, then
    \[
        (k,k+2) \cdot \alpha = (k,k+1)(k+1,k+2)(k,k+1) \cdot \alpha = \alpha,
    \]
    so all the cohomology groups vanish.
    By the same argument, all the cohomology groups vanish when $d < n-k+1$, and one can proceed from
    \[
        \alpha' = (k, k+1, \dots, n) \cdot \alpha = (i_1, \dots, i_{k-1}, (d-1)^{n-k}, n-k)
    \]
    with $d \geq n-k+1$.
    If $\alpha'$ itself is a partition, i.e., $i_1 = \cdots = i_{k-1} = n-k$ and $d = n-k+1$, then $H^i(K_{((n-k)^{k-1})} \Ucal^{\vee}(-(n-k+1))) \neq 0$ if and only if $i = n-k$, and in this case
    \[
        H^{n-k}(K_{((n-k)^{k-1})} \Ucal^{\vee}(-(n-k+1))) = (\det E)^{\otimes n-k} \simeq \CC.
    \]
    
    Similarly, one can show that the cohomology groups do not vanish except for the top ones if and only if $i_1 = \dots = i_{\ell} = n-k$ and $i_{\ell+1} = \dots = i_{k} = 0$ for some $0 \leq \ell < k$, and in this case
    \[
        H^{(n-k)(k-\ell)}(K_{((n-k)^{\ell})} \Ucal^{\vee}(-(n-\ell))) = (\det E)^{\otimes n-k} \simeq \CC.
    \]
    On the other hand, if $i_{\ell-1}-i_{\ell} < n-k$ for any $\ell$, one obtains a partition
    \[
        \sigma \cdot \alpha = ((d-k)^{n-k}, i_1+n-k, \dots, i_{k}+n-k)
    \]
    if and only if $d \geq n+i_1$, where
    \[
        \sigma = (1, 2, \dots, n-k+1)(2,3, \dots, n-k+2) \dots (k, k+1, \dots, n).
    \]
    One may observe that $\ell(\sigma') = k(n-k)$.

    For the second statement, notice that
    \[
        \Vcal((0^{k}, j_1+d, \dots, j_{n-k}+d)) = K_{(j_1+d, \dots, j_{n-k}+d)} \Qcal^{\vee} = K_{(j_1, \dots, j_{n-k})} \Qcal^{\vee}(-d).
    \]
    Let $\beta = (0^{k}, j_1+d, \dots, j_{n-k}+d)$.
    Again, we aim to find $\sigma \in \Sigma_n$ with $\sigma \cdot \beta$ nonincreasing.
    When $\beta$ itself is a partition, i.e., $d = 0$ and $j_1 = \dots = j_{n-k} = 0$, then $H^i(K_{\mathrm{(0)}}\Qcal^{\vee}) = H^{i}(\Ocal) \neq 0$ if and only if $i = 0$, and in this case $H^0(\Ocal) = \CC$.
    If $\beta$ is not a partition, then we first get rid of zeros at the first $k$-th entries.
    As above, all the cohomology groups vanish when $d < k-j_1+1$, and one can proceed from
    \[
        \beta' = (1, 2, \dots, k+1) \cdot \beta = (j_1+d-k, 1^k, j_2+d, \dots, j_{n-1}+d)
    \]
    with $d \geq k-j_1+1$.
    If $\beta'$ itself is a partition, i.e., $d = 1$, $j_1 = k$ and $j_2 = \dots = j_{n-k} = 0$, then $H^i(K_{(k)}\Qcal^{\vee}(-1)) = H^i(\mathrm{Sym}^k \Qcal^{\vee}(-1)) \neq 0$ if and only if $i = k$, and in this case
    \[
        H^k(\mathrm{Sym}^k\Qcal^{\vee}(-1)) = \det E \simeq \CC.
    \]
    
    In the same fashion, one can prove that the cohomology groups do not vanish except for the top ones if and only if $j_1 = \dots = j_{\ell} = k$ and $j_{\ell+1} = \dots = j_{n-k} = 0$ for some $0 \leq \ell < n-k$, and in this case
    \[
        H^{k\ell}(K_{(k^{\ell})}\Qcal^{\vee}(-\ell)) = (\det E)^{\otimes \ell} \simeq \CC
    \]
    is the only nonzero cohomology group.
    On the other hand, if $j_{\ell-1} - j_{\ell} < k$ for any $\ell$, one obtains a partition
    \[
        \sigma \cdot \beta = (j_1+d-k, \dots, j_{n-k}+d-k, (n-k)^k)
    \]
    if and only if $d \geq n$, where
    \[
        \sigma = (n-k, n-k+1, \dots, n)(n-k-1, n-k, \dots, n-1) \dots (1, 2, \dots, k+1).
    \]
    One can observe that $\ell(\sigma) = k(n-k)$, whence the assertion.
\end{proof}

%---------------------------------------------------------------------

\section{Proof of main theorem~\ref{thm:main-theorem}}\label{sec:main-results}

In this section, we prove our main theorem \ref{thm:main-theorem}.
Here, the assumption that Fano threefolds are general is necessary.
Indeed, we use Mukai's classification of general Fano threefolds as the zero loci of regular sections of homogeneous vector bundles on the product of (weighted) Grassmannians (cf. \cite{Mukai1989}, see also \cite{CCGK2016} and \cite{DBFT2021}).
We also write Iskovskikh's classical description of Fano threefolds with Picard number one for the comparison (cf. \cite{Iskovskikh1978}).
Observe that del Pezzo surfaces are uniruled and K3 surfaces have infinitely many rational curves by \cite[Theorem~A]{CGL2022}, so in particular they are not algebraically hyperbolic.

From now on, let $S \in |\Ocal_X(a)|$ be a very general surface with $a \geq r+1$.
It is easy to see that the Assumptions \hyperlink{Assump1}{\textbf{1}} and \hyperlink{Assump2}{\textbf{2}} are satisfied throughout all the deformation types of $X$.
Since $S$ is very general, we have $\mathrm{Pic}(S) \simeq \ZZ\langle \Ocal_S(1) \rangle$ by the generalized Noether-Lefschetz theorem \cite[Theorem~1]{RS2009}.
Let $f:C \rightarrow S$ is a map from a smooth curve $C$ of genus $g$ that is birational onto its image.

\vspace{8pt} \noindent \hypertarget{Fano1-1}{\textbf{Fano 1-1.}} Double cover of $\PP^3$ branched over a sextic surface. \\
\noindent \textit{Description:} $\Lfr(\Ocal(6)) \subset \PP(1^4, 3)$.

Note that for a weighted projective space $w\PP$, the intermediate cohomology group $H^i(w\PP, \Ocal_{w\PP}(d))$ always vanishes for $0 < i < \dim w\PP$ and $d \in \ZZ$ by \cite[p.39, Theorem]{Dolgachev1982}, so the condition~\ref{eqn:restriction-map-is-surjective} (or the Assumption \hyperlink{Assump3}{\textbf{3}}) is satisfied.

We have an action of $G = \Hom(\Sym^3 \CC^4, \CC) \rtimes \GL_4$ on $A = \PP(1^4,3)$ as in Section~\ref{subsec:weighted-projective-spaces}.
This $G$-action has a dense orbit $A_0 = A \setminus \{ [0^4:1] \}$.
By Proposition~\ref{prop:basic-diagrams}, we have a surjection $(M_{\Ocal_{A}(6)} \oplus M_{\Ocal_{A}(a)})|_C \rightarrow N_{f/S}$.
If $a \geq 3$, the line bundle $\Ocal_{A}(3)$ is section-dominating for $\Ocal_{A}(a)$ on $A_0$ by Theorem~\ref{thm:Almost-section-dominating-on-WPS}, which yields a generic surjection $M_{\Ocal_{A}(3)}|_C \rightarrow N_{f/S}$ as $N_{f/S}$ has rank one.
From Proposition~\ref{prop:Mukai-bundle-degree}, we infer that
 \[
  2g-2 \geq (C.K_S)- \deg_C \Ocal_{A}(3) = (a-4) \deg_C \Ocal_A(1).
 \]
Hence we obtain the desired estimate provided $a > 4$.

For $a = 3$ or $4$, the sheaf $M_{\Ocal_{A}(3)}|_C$ admits a rank one quotient $Q \subset N_{f/S}$.
Say $Q$ has degree $q$ with respect to $\Ocal_A(3)$.
Consider the embedding $A \hookrightarrow \PP^n$ induced by the very ample line bundle $\Ocal_A(3)$.
By the scroll argument \ref{lem:surface-scroll}, there is a surface scroll $\Sigma \subset \PP^n$ whose $\Ocal_{\PP^n}(1)$-degree is equal to
\[
    q + \deg_C \Ocal_{\PP^n}(1) = q + 3\deg_C \Ocal_A(1).
\]

Note that there is a quadric $Z \in |\Ocal_{\PP^n}(2)|$ with $X = A \cap Z$ as the restriction map $H^0(\PP^n, \Ocal_{\PP^n}(2)) \rightarrow H^0(A, \Ocal_A(6))$ is surjective.
If $\Sigma \not\subset Z$, from $C \subset Z \cap \Sigma$ we have
\[
    2(q+3 \deg_C \Ocal_A(1)) \geq 3\deg_C \Ocal_A(1)
\]
or
\[
    q \geq -\frac{3}{2} \deg_C \Ocal_A(1).
\]
Thus we have
\[
    2g(C)-2 \geq (C.K_S) + q \geq \left( a - \frac{5}{2} \right) \deg_C \Ocal_A(1).
\]
This yields the desired estimate when $a \geq 3$.
The computation performed using Sage confirms that $A$ is defined by quadrics in $\PP^n$, so the same argument can be applied when $\Sigma \not\subset A$.

Now assume that $\Sigma \subset X$.
Since $S$ is of general type, it is not a scroll, so $S \neq \Sigma$.
Thus from $C \subset S \cap \Sigma$, we infer that
\[
    \frac{a}{3}(q+3\deg_C \Ocal_{A}(1)) \geq 3\deg_C \Ocal_{A}(1).
\]
Hence we get
\[
    2g(C)-2 \geq \left( a + \frac{9}{a} - 4 \right) \deg_C \Ocal_A(1) \geq 2 \deg_C \Ocal_A(1).
\]
In conclusion, $S$ is algebraically hyperbolic if $a \geq 3$.
The remaining case is $a = 2$.

\vspace{8pt} \noindent \hypertarget{Fano1-2}{\textbf{Fano 1-2.}} Quartic threefold. \\
\noindent \textit{Description:} $\Lfr(\Ocal(4)) \subset \PP^4$.

Since $\Ocal_A(1)$ is section-dominating for $\Ocal_A(k)$ if $k \geq 1$, we have
 \[
  2g-2 \geq (C.K_S) - \deg_C \Ocal_A(1) = (a-2) \deg_C \Ocal_A(1).
 \]
This gives the desired estimate when $a > 2$.

The sheaf $M_{\Ocal_A(1)}|_C$ admits a rank one quotient $Q \subset N_{f/S}$ of degree $q$ with respect to $\Ocal_A(1)$.
By the scroll argument \ref{lem:surface-scroll}, there is a surface scroll $\Sigma \subset A$ whose $\Ocal_A(1)$-degree is $q + \deg_C \Ocal_A(1)$.
Since $S$ is of general type, $S$ is not equal to $\Sigma$.
Writing $S = Z \cap Z'$ for some $Z \in |\Ocal_A(4)|$ and $Z' \in |\Ocal_A(a)|$, the scroll $\Sigma$ is not contained either in $Z$ or in $Z'$.
In either case, we have an estimate
\[
    k(q + \deg_C \Ocal_A(1)) \geq \deg_C \Ocal_A(1)
\]
for $k = 4$ or $a$.
Thus we have
\[
    2g-2 \geq \left( a + \frac{1}{k} - 2 \right) \deg_C \Ocal_A(1)
\]
as desired.
We see that $S$ is algebraically hyperbolic if and only if $a \geq 2$.

\vspace{8pt} \noindent \hypertarget{Fano1-3}{\textbf{Fano 1-3.}} Complete intersection of a quadric and a cubic in $\PP^5$. \\
\noindent \textit{Description:} $\Lfr(\Ocal(2) \oplus \Ocal(3)) \subset \PP^5$.

The same argument as in \hyperlink{Fano1-2}{\textbf{1-2}} works.

\vspace{8pt} \noindent \hypertarget{Fano1-4}{\textbf{Fano 1-4.}} Complete intersection of three quadrics in $\PP^6$. \\
\noindent \textit{Description:} $\Lfr(\Ocal(2)^{\oplus 3}) \subset \PP^6$.

The same argument as in \hyperlink{Fano1-2}{\textbf{1-2}} works.

\vspace{8pt} \noindent \hypertarget{Fano1-5}{\textbf{Fano 1-5.}} Section of Pl\"ucker embedding of $\Gr(2,5)$ by codimension 2 subspace and a quadric. \\
\noindent \textit{Description:} $\Lfr(\Ocal(2) \oplus \Ocal(1)^{\oplus 2}) \subset \Gr(2,5)$.

Note that for the Grassmannian $\Gr(k,n)$, the intermediate cohomology groups $H^i(\Gr(k,n), \Ocal(d))$ always vanish for $0 < i < k(n-k)$ and $d \in \ZZ$ by Theorem~\ref{thm:Bott-theorem-for-Grassmannians}, so the condition~\ref{eqn:restriction-map-is-surjective} is safisfied.
Since the Pl\"ucker embedding of the Grassmannian is projectively normal, the same argument as above yields the desired estimate when $a > 2$.

For $a = 2$, consider the Pl\"ucker embedding $A \hookrightarrow \PP^n$.
Then there is a surface scroll $\Sigma \subset \PP^n$ as before.
If $\Sigma \not\subset A$, since $A$ is defined by quadrics in $\PP^n$, there is a quadric $Z \in |\Ocal_{\PP^n}(2)|$ with $\Sigma \not\subset Z$.
In this case, we have
\[
    2(q+\deg_C \Ocal_A(1)) \geq \deg_C \Ocal_A(1).
\]
Thus we have
\[
    2g - 2 \geq \left( a - \frac{3}{2} \right) \deg_C \Ocal_A(1).
\]
If $\Sigma \subset A$, one can argue as in \hyperlink{Fano1-2}{\textbf{1-2}}.
Thus $S$ is algebraically hyperbolic if and only if $a \geq 2$.

\vspace{8pt} \noindent \hypertarget{Fano1-6}{\textbf{Fano 1-6.}} $X_{12} \subset \PP^8$. \\
\noindent \textit{Description:} $\Lfr(\Ucal^{\vee}(1) \oplus \Ocal(1)) \subset \Gr(2,5)$.

Since
\[
    \Fcal^{\vee} = (\Ucal^{\vee}(1) \oplus \Ocal(1))^{\vee} = \Ucal^{\vee}(-2) \oplus \Ocal(-1),
\]
from Bott's theorem~\ref{thm:Bott-theorem-for-Grassmannians} and Theorem~\ref{thm:Cohomology-vanishing-of-homogeneous}, one can show that that the condition~\ref{eqn:restriction-map-is-surjective} holds.

We prove that $\Ocal(1)$ is section-dominating for $\Ucal^{\vee}(1)$.
Consider the exact sequence
\[
\begin{tikzcd}
    0 \ar[r] & \Qcal^{\vee} \ar[r] & \Ocal^{\oplus 5} \ar[r] & \Ucal^{\vee} \ar[r] & 0.
\end{tikzcd}
\]
By taking the tensor product by $\Ocal(1)$, we get a diagram
\[
\begin{tikzcd}
    & & 0 \ar[d] & 0 \ar[d] & \\
    & & H^0(\Qcal^{\vee}(1)) \otimes \Ocal \ar[r] \ar[d] & \Qcal^{\vee}(1) \ar[d] & \\
    0 \ar[r] & M_{\Ocal(1)}^{\oplus 5} \ar[r] \ar[d] & H^0(\Ocal(1))^{\oplus 5} \otimes \Ocal \ar[r] \ar[d] & \Ocal(1)^{\oplus 5} \ar[r] \ar[d] & 0 \\
    0 \ar[r] & M_{\Ucal^{\vee}(1)} \ar[r] & H^0(\Ucal^{\vee}(1)) \otimes \Ocal \ar[r] \ar[d] & \Ucal^{\vee}(1) \ar[r] \ar[d] & 0 \\
    & & 0 & 0 &
\end{tikzcd}
\]
Since
\[
    \Vcal((1,1,1)) = K_{(1,1)}\Ucal^{\vee} \otimes K_{(1)}\Qcal^{\vee} = \Qcal^{\vee}(1),
\]
we have $H^1(\Qcal^{\vee}(1)) = 0$, so the middle column is exact.
Note that $\Qcal^{\vee}(1) = \bigwedge^2 \Qcal$ is globally generated, so the map in the first row is surjective.
Hence we have a surjection $M_{\Ocal(1)}^{\oplus 5} \rightarrow M_{\Ucal^{\vee}(1)}$ by the snake lemma.
In other words, $\Ocal(1)$ is section-dominating for $\Ucal^{\vee}(1)$.
In conclusion, combined with the scroll argument, $S$ is algebraically hyperbolic if and only if $a \geq 2$.

\vspace{8pt} \noindent \hypertarget{Fano1-7}{\textbf{Fano 1-7.}} Section of Pl\"ucker embedding of $\Gr(2,6)$ by codimension 5 subspace. \\
\noindent \textit{Description:} $\Lfr(\Ocal(1)^{\oplus 5}) \subset \Gr(2,6)$.

The same argument as in \hyperlink{Fano1-5}{\textbf{1-5}} works.

\vspace{8pt} \noindent \hypertarget{Fano1-8}{\textbf{Fano 1-8.}} $X_{16} \subset \PP^{10}$. \\
\noindent \textit{Description:} $\Lfr(\bigwedge^2 \Ucal^{\vee} \oplus \Ocal(1)^{\oplus 3}) \subset \Gr(3,6)$.

Since
\[
    \Fcal^{\vee} = \left( \bigwedge^2 \Ucal^{\vee} \oplus \Ocal(1)^{\oplus 3} \right)^{\vee} = \Ucal^{\vee}(-1) \oplus \Ocal(-1)^{\oplus 3},
\]
from Bott's theorem~\ref{thm:Bott-theorem-for-Grassmannians} and Theorem~\ref{thm:Cohomology-vanishing-of-homogeneous}, we infer that the condition~\ref{eqn:restriction-map-is-surjective} holds.

There is an obstruction to estimating the degree of $N_{f/S}$ when it admits a generic surjection from $M_{\bigwedge^2 \Ucal^{\vee}}|_C$ and no other components.
Note that
 \[
  \Vcal((1^2,0,1^3)) = K_{(1^2)} \Ucal^{\vee} \otimes K_{(1^3)} \Qcal^{\vee}
    = \bigwedge^2 \Ucal^{\vee}(-1).
 \]
Since the transposition $(3,4)$ fixes $(1^2,0,1^3)$, from Bott's theorem \ref{thm:Bott-theorem-for-Grassmannians} we have $H^i(\bigwedge^2 \Ucal^{\vee}(-1)) = 0$ for all $i$.
Hence $\Ocal(1)$ is not section-dominating for $\bigwedge^2 \Ucal^{\vee}$.

Instead, we claim that $\Ucal^{\vee}$ is section-dominating for $\bigwedge^2 \Ucal^{\vee}$.
Consider the dual of the Koszul complex
\[
\begin{tikzcd}
    0 \ar[r] & \mathrm{Sym}^2\Qcal^{\vee} \ar[r] & \Ocal^{\oplus 21} \ar[r] & (\Ucal^{\vee})^{\oplus 6} \ar[r] & \bigwedge^2 \Ucal^{\vee} \ar[r] & 0.
\end{tikzcd}
\]
Let $\Kcal$ be the kernel of the last map.
As in \hyperlink{Fano1-6}{\textbf{1-6}}, we have a diagram
\[
\begin{tikzcd}
    & & 0 \ar[d] & 0 \ar[d] & \\
    & & H^0(\Kcal) \otimes \Ocal \ar[r] \ar[d] & \Kcal \ar[d] & \\
    0 \ar[r] & M_{\Ucal^{\vee}}^{\oplus 6} = (\Qcal^{\vee})^{\oplus 6} \ar[r] \ar[d] & H^0(\Ucal^{\vee})^{\oplus 6} \otimes \Ocal \ar[r] \ar[d] & (\Ucal^{\vee})^{\oplus 6} \ar[r] \ar[d] & 0 \\
    0 \ar[r] & M_{\bigwedge^2 \Ucal^{\vee}} \ar[r] & H^0 \left( \bigwedge^2 \Ucal^{\vee} \right) \otimes \Ocal \ar[r] \ar[d] & \bigwedge^2 \Ucal^{\vee} \ar[r] \ar[d] & 0 \\
    & & 0 & 0 &
\end{tikzcd}
\]
Note that $H^1(\Ocal) = H^2(\mathrm{Sym}^2\Qcal^{\vee}) = 0$ by Theorem~\ref{thm:Cohomology-vanishing-of-homogeneous}.
Thus $H^1(\Kcal) = 0$, i.e., the middle column is exact.
Also, since $\Kcal$ admits a surjection from the trivial bundle, it is globally generated.
Hence we have a surjection $(\Qcal^{\vee})^{\oplus 6} \rightarrow M_{\bigwedge^2 \Ucal^{\vee}}$ by the snake lemma, as wished.

Let $H$ be a hyperplane section of $\Gr(3,6)$.
We prove that the $\QQ$-vector bundle $\Qcal^{\vee} \langle \frac{H}{2} \rangle$ is nef on $C$.
Indeed, if there is a generic surjection $M_{\bigwedge^2 \Ucal^{\vee}}|_C \rightarrow N_{f/S}$, then it induces a generic surjection $\Qcal^{\vee}|_C \rightarrow N_{f/S}$.
By the assertion, we have the degree bound
\[
    \deg N_{f/S} \geq -\frac{1}{2} \deg_C \Ocal(1),
\]
or
\[
    2g-2 \geq (C.K_S) - \frac{1}{2} \deg_C \Ocal(1) = \left( a - \frac{3}{2} \right) \deg_C \Ocal(1).
\]
Combined with the scroll argument, one can see that $S$ is algebraically hyperbolic if and only if $a \geq 2$.

We only need to prove that $\Qcal^{\vee} \langle \frac{H}{2} \rangle$ is nef on $C' = f(C)$.
If $\Qcal^{\vee}|_{C'} \rightarrow N$ is a surjection onto a line bundle $N$ on $C'$, letting $\Kcal$ be the kernel, the assertion is equivalent to saying that
\[
    \deg_{C'} N = \deg_{C'} \Qcal^{\vee} - \deg_{C'} \Kcal \geq -\frac{1}{2} \deg_{C'} \Ocal(1).
\]
If $\deg_{C'} \Kcal \geq -\frac{1}{2} \deg_{C'} \Ocal(1)$, the inclusion $\Kcal \subset \Qcal^{\vee}|_{C'}$ induces a nonzero section of $\bigwedge^2 \Qcal^{\vee}|_{C'} \otimes \det \Kcal^{\vee}$, or
\[
    \mathrm{Sym}^{2}\left( \bigwedge^2 \Qcal^{\vee}|_{C'} \otimes \det \Kcal^{\vee} \right) \subset \left. \mathrm{Sym}^2 \left( \bigwedge^2 \Qcal^{\vee} \right)(1) \right|_{C'}.
\]
Hence we only need to show that
\begin{equation} \label{eqn:Vanishing-condition-on-curve-1-8}
    H^0 \left( C', \mathrm{Sym}^2 \left( \bigwedge^2 \Qcal^{\vee} \right) \otimes \Ocal(1) \right) = 0.
\end{equation}

For the vanishing result, we first prove that the restriction map $H^0(A,\Ocal(d)) \rightarrow H^0(S,\Ocal(d))$ is surjective for $d \geq 0$.
This can be done similarly as in the beginning of Section~\ref{subsec:setup}; we only need to replace $\Fcal$ by
\[
    \Ecal = \bigwedge^2 \Ucal^{\vee} \oplus \Ocal(1)^{\oplus 3} \oplus \Ocal(a).
\]
In particular, there is the Koszul resolution
\[
\begin{tikzcd}
    K^{\bullet}(\Ecal^{\vee}) \ar[r] & \Ocal_{C'} \ar[r] & 0,
\end{tikzcd}
\]
where
\[
    \Ecal = \bigwedge^{2} \Ucal^{\vee} \oplus \Ocal(1)^{\oplus 3} \oplus \Ocal(a) \oplus \Ocal(d)
\]
for an integer $d \geq 1$ with $C' \in |\Ocal_S(d')|$, and
\[
    K^{i}(\Ecal^{\vee}) = \bigwedge^{-i-1} \Ecal^{\vee}
\]
for $i \leq -1$.

Note that $\mathrm{Sym}^{2}(\bigwedge^2 \Qcal^{\vee})$ is a direct summand of $(\bigwedge^2 \Qcal^{\vee})^{\otimes 2}$ and there is a decomposition of $(\bigwedge^2 \Qcal^{\vee})^{\otimes 2}$ into irreducible representations
\[
\begin{tikzcd}
    0 \ar[r] & \Qcal^{\vee}(-1) \ar[r] & \bigwedge^2 \Qcal^{\vee} \otimes \bigwedge^2 \Qcal^{\vee} \ar[r] & K_{(2,2)}\Qcal^{\vee} \ar[r] & 0.
\end{tikzcd}
\]
From the spectral sequence, it suffices to show that
\begin{equation} \label{eqn:Spectral-sequence-1-8}
    H^{i}(K^{-i-1}(\Ecal^{\vee}) \otimes \Qcal^{\vee}) = H^{i}(K^{-i-1}(\Ecal^{\vee}) \otimes K_{(2,2)}\Qcal^{\vee}(1)) = 0.
\end{equation}

To this end, we need several cohomology vanishing results.
We made a table of necessary vanishing results below.
For example, the first row reads as the vanishing of $H^{i}(\Qcal^{\vee}(-N))$ is required, where $i \leq 8$ and $N \geq 0$.
\begin{table}[ht]
\begin{tabular}{|c|c|c|}
\hline
$H^i$      & vector bundle                                & twist       \\ \hline
$i \leq 8$ & $\Qcal^{\vee}$                               & $N \geq 0$  \\ \hline
$i \leq 8$ & $K_{(2,2)}\Qcal^{\vee}$                      & $N \geq -1$ \\ \hline
$i \leq 6$ & $\Ucal^{\vee} \otimes \Qcal^{\vee}$          & $N \geq 1$  \\ \hline
$i \leq 6$ & $\Ucal^{\vee} \otimes K_{(2,2)}\Qcal^{\vee}$ & $N \geq 0$  \\ \hline
$i \leq 7$ & $K_{(1,1)}\Ucal^{\vee} \otimes \Qcal^{\vee}$ & $N \geq 2$  \\ \hline
\multirow{2}{*}{$i \leq 7$} & \multirow{2}{*}{$K_{(1,1)}\Ucal^{\vee} \otimes K_{(2,2)}\Qcal^{\vee}$} & $N \geq 1$ \\
           & & $N \geq 7$ if $i = 7$ \\ \hline
\end{tabular}
\end{table}

\noindent The first two vanishing results are due to Theorem~\ref{thm:Cohomology-vanishing-of-homogeneous} and the fact that
\[
    \Vcal((1,1,1,2,2,0)) = K_{(2,2)}\Qcal^{\vee}(1)
\]
has vanishing cohomology groups by Theorem~\ref{thm:Bott-theorem-for-Grassmannians}.
For the other vanishing results, below we listed the non-vanishing results of cohomology groups, except for the top ones.
For example, the first row reads as follows: for $i \leq 8$ and $N \geq 0$, $H^i(\Ucal^{\vee} \otimes K_{(2,2)}\Qcal^{\vee}(-N)) \neq 0$ if and only if $i = 8$ and $N = 5$.
\begin{table}[ht]
\begin{tabular}{|c|c|c|}
\hline
vector bundle                                & ($i$, twist)       \\ \hline
$\Ucal^{\vee} \otimes \Qcal^{\vee}$          & -        \\ \hline
$\Ucal^{\vee} \otimes K_{(2,2)}\Qcal^{\vee}$ & $(8,5)$  \\ \hline
$K_{(1,1)}\Ucal^{\vee} \otimes \Qcal^{\vee}$ & $(1,1)$  \\ \hline
$K_{(1,1)}\Ucal^{\vee} \otimes \Qcal^{\vee}$ & $(7,4)$  \\ \hline
\end{tabular}
\end{table}

\newpage

\vspace{8pt} \noindent \hypertarget{Fano1-9}{\textbf{Fano 1-9.}} $X_{18} \subset \PP^{11}$. \\
\noindent \textit{Description:} $\Lfr(\Qcal^{\vee}(1) \oplus \Ocal(1)^{\oplus 2}) \subset \Gr(2,7)$.

Note that
\[
    \Fcal^{\vee} = (\Qcal^{\vee}(1) \oplus \Ocal(1)^{\oplus 2})^{\vee} = \Qcal(-1) \oplus \Ocal(-1)^{\oplus 2}
\]
and
\[
    \bigwedge^i \Qcal = \bigwedge^{5-i} \Qcal^{\vee}(1)
\]
for $1 \leq i \leq 4$.
Again, it is easy to see that the condition~\ref{eqn:restriction-map-is-surjective} is satisfied.

Consider the dual of the Koszul complex
\[
\begin{tikzcd}
    0 \ar[r] & \mathrm{Sym}^4\Ucal \ar[r] & \Ocal^{\oplus 210} \ar[r] & \Qcal^{\oplus 84} \\
    \ar[r] & \left( \bigwedge^2 \Qcal \right)^{\oplus 28} \ar[r] & \left( \bigwedge^3 \Qcal \right)^{\oplus 7} \ar[r] & \bigwedge^4 \Qcal \ar[r] & 0.
\end{tikzcd}
\]
Let $\Kcal$ be the kernel of the last map.
By Theorem~\ref{thm:Bott-theorem-for-Grassmannians} and the spectral sequence, the map
\[
\begin{tikzcd}
    H^0 \left( \bigwedge^3 \Qcal \right)^{\oplus 7} \ar[r] & H^0 \left( \bigwedge^4 \Qcal \right)
\end{tikzcd}
\]
is surjective. Thus we have a diagram
\[
\begin{tikzcd}
    & & 0 \ar[d] & 0 \ar[d] & \\
    & & H^0(\Kcal) \otimes \Ocal \ar[r] \ar[d] & \Kcal \ar[d] & \\
    0 \ar[r] & M_{\bigwedge^3 \Qcal}^{\oplus 7} \ar[r] \ar[d] & H^0\left( \bigwedge^3 \Qcal \right)^{\oplus 7} \otimes \Ocal \ar[r] \ar[d] & \left( \bigwedge^3 \Qcal \right)^{\oplus 7} \ar[r] \ar[d] & 0 \\
    0 \ar[r] & M_{\bigwedge^4 \Qcal} \ar[r] & H^0 \left( \bigwedge^4 \Qcal \right) \otimes \Ocal \ar[r] \ar[d] & \bigwedge^4 \Qcal \ar[r] \ar[d] & 0 \\
    & & 0 & 0 &
\end{tikzcd}
\]
Since $\Kcal$ admits a surjection from $\bigwedge^2 \Qcal$, it is globally generated.
Hence $\bigwedge^3 \Qcal$ is section-dominating for $\bigwedge^4 \Qcal$.
By the same procedure, one can show that $\Qcal$ is section-dominating for $\bigwedge^4 \Qcal$.

Observe that the restriction map $H^0(A,\Ocal(d)) \rightarrow H^0(S,\Ocal(d))$ is surjective as in \hyperlink{Fano1-8}{\textbf{1-8}}.
We aim to show that
\begin{equation}
    H^0(f(C), K_{(2)}\Ucal^{\vee}(-1)) = 0
\end{equation}
because $M_{\Qcal} = \Ucal$.
Below, we listed the required vanishing results and the non-vanishing results of cohomology groups, except for the top ones.
In conclusion, $S$ is algebraically hyperbolic if and only if $a \geq 2$.

\begin{table}[ht]
\begin{tabular}{|c|c|c|}
\hline
$H^i$      & vector bundle                                & twist       \\ \hline
$i \leq 9$ & $K_{(2)}\Ucal^{\vee}$                        & $N \geq 1$  \\ \hline
$i \leq 5$ & $K_{(2)}\Ucal^{\vee} \otimes K_{(1^4)}\Qcal^{\vee}$ & $N \geq 1$ \\ \hline
$i \leq 6$ & $K_{(2)}\Ucal^{\vee} \otimes K_{(1^3)}\Qcal^{\vee}$ & $N \geq 2$  \\ \hline
$i \leq 7$ & $K_{(2)}\Ucal^{\vee} \otimes K_{(1^2)}\Qcal^{\vee}$ & $N \geq 3$  \\ \hline
$i \leq 8$ & $K_{(2)}\Ucal^{\vee} \otimes \Qcal^{\vee}$   & $N \geq 4$  \\ \hline
\end{tabular}
\end{table}

\begin{table}[ht]
\begin{tabular}{|c|c|c|}
\hline
vector bundle                                       & ($i$, twist) \\ \hline
$K_{(2)}\Ucal^{\vee} \otimes K_{(1^4)}\Qcal^{\vee}$ & $(9,7)$      \\ \hline
$K_{(2)}\Ucal^{\vee} \otimes K_{(1^3)}\Qcal^{\vee}$ & $(8,6)$      \\ \hline
$K_{(2)}\Ucal^{\vee} \otimes K_{(1^2)}\Qcal^{\vee}$ & $(2,2)$      \\ \hline
$K_{(2)}\Ucal^{\vee} \otimes \Qcal^{\vee}$          & $(1,1)$      \\ \hline
\end{tabular}
\end{table}

\vspace{8pt} \noindent \hypertarget{Fano1-10}{\textbf{Fano 1-10.}} $X_{22} \subset \PP^{13}$. \\
\noindent \textit{Description:} $\Lfr((\bigwedge^2 \Ucal^{\vee})^{\oplus 3}) \subset \Gr(3,7)$.

Since
\[
    \Fcal^{\vee} = \left( \left( \bigwedge^2 \Ucal^{\vee} \right)^{\oplus 3} \right)^{\vee} = (\Ucal^{\vee}(-1))^{\oplus 3},
\]
we infer that the condition~\ref{eqn:restriction-map-is-surjective} is satisfied as before.
By a similar argument, the restriction map $H^0(A, \Ocal(d)) \rightarrow H^0(S, \Ocal(d))$ is surjective as well.

As in \hyperlink{Fano1-8}{\textbf{1-8}}, we obtain the algebraic hyperbolicity of $S$ once we prove
\[
    H^0 \left(f(C), \mathrm{Sym}^2 \left( \bigwedge^3 \Qcal^{\vee} \right) \otimes \Ocal(1) \right) = 0.
\]
Consider a decomposition into irreducible representations:
\[
\begin{tikzcd}
    0 \ar[r] & \bigwedge^2 \Qcal^{\vee}(-1) \ar[r] & \left( \bigwedge^3 \Qcal^{\vee} \right)^{\otimes 2} \ar[r] & K_{(2^3)} \Qcal^{\vee} \ar[r] & 0.
\end{tikzcd}
\]
Then we only need to see that
\begin{equation}
    H^0 (f(C), K_{(2^3)}\Qcal^{\vee}(1)) = H^0(f(C), K_{(1^2)}\Qcal^{\vee}) = 0.
\end{equation}
As before, below we listed the required vanishing results and the non-vanishing results of cohomology groups, except for the top ones.
In conclusion, $S$ is algebraically hyperbolic if and only if $a \geq 2$.

\begin{table}[ht]
\begin{tabular}{|c|c|c|}
\hline
$H^i$       & vector bundle                                & twist       \\ \hline
$i \leq 11$ & $K_{(2^3)}\Qcal^{\vee}$ & $N \geq -1$  \\ \hline
$i \leq 11$ & $K_{(1^2)}\Qcal^{\vee}$ & $N \geq 0$  \\ \hline
$i \leq 9$  & $\Ucal^{\vee} \otimes K_{(2^3)}\Qcal^{\vee}$ & $N \geq 0$  \\ \hline
$i \leq 9$  & $\Ucal^{\vee} \otimes K_{(1^2)}\Qcal^{\vee}$ & $N \geq 1$  \\ \hline
\multirow{2}{*}{$i \leq 10$} & \multirow{2}{*}{$K_{(1,1)}\Ucal^{\vee} \otimes K_{(2^3)}\Qcal^{\vee}$} & $N \geq 1$ \\
           & & $N \geq 7$ if $i = 10$ \\ \hline
$i \leq 10$ & $K_{(1,1)}\Ucal^{\vee} \otimes K_{(1^2)}\Qcal^{\vee}$ & $N \geq 2$  \\ \hline
$i \leq 7$ & $K_{(2)}\Ucal^{\vee} \otimes K_{(2^3)}\Qcal^{\vee}$ & $N \geq 1$  \\ \hline
$i \leq 7$  & $K_{(2)}\Ucal^{\vee} \otimes K_{(1^2)}\Qcal^{\vee}$   & $N \geq 2$  \\ \hline
$i \leq 8$  & $K_{(2,1)}\Ucal^{\vee} \otimes K_{(2^3)}\Qcal^{\vee}$ & $N \geq 2$  \\ \hline
$i \leq 8$  & $K_{(2,1)}\Ucal^{\vee} \otimes K_{(1^2)}\Qcal^{\vee}$ & $N \geq 3$  \\ \hline
$i \leq 9$  & $K_{(2,2)}\Ucal^{\vee} \otimes K_{(2^3)}\Qcal^{\vee}$ & $N \geq 3$  \\ \hline
$i \leq 9$  & $K_{(2,2)}\Ucal^{\vee} \otimes K_{(1^2)}\Qcal^{\vee}$ & $N \geq 4$  \\ \hline
$i \leq 5$  & $K_{(3)}\Ucal^{\vee} \otimes K_{(2^3)}\Qcal^{\vee}$   & $N \geq 2$  \\ \hline
$i \leq 5$  & $K_{(3)}\Ucal^{\vee} \otimes K_{(1^2)}\Qcal^{\vee}$   & $N \geq 3$  \\ \hline
$i \leq 6$  & $K_{(3,1)}\Ucal^{\vee} \otimes K_{(2^3)}\Qcal^{\vee}$ & $N \geq 3$  \\ \hline
$i \leq 6$  & $K_{(3,1)}\Ucal^{\vee} \otimes K_{(1^2)}\Qcal^{\vee}$ & $N \geq 4$  \\ \hline
$i \leq 7$  & $K_{(3,2)}\Ucal^{\vee} \otimes K_{(2^3)}\Qcal^{\vee}$ & $N \geq 4$  \\ \hline
$i \leq 7$  & $K_{(3,2)}\Ucal^{\vee} \otimes K_{(1^2)}\Qcal^{\vee}$ & $N \geq 5$  \\ \hline
$i \leq 8$  & $K_{(3,3)}\Ucal^{\vee} \otimes K_{(2^3)}\Qcal^{\vee}$ & $N \geq 5$  \\ \hline
$i \leq 8$  & $K_{(3,3)}\Ucal^{\vee} \otimes K_{(1^2)}\Qcal^{\vee}$ & $N \geq 6$  \\ \hline
\end{tabular}
\end{table}

\begin{table}[ht]
\begin{tabular}{|c|c|c|}
\hline
vector bundle                                         & ($i$, twist) \\ \hline
$\Ucal^{\vee} \otimes K_{(2^3)}\Qcal^{\vee}$          & $(11,6)$     \\ \hline
$\Ucal^{\vee} \otimes K_{(1^2)}\Qcal^{\vee}$          & -            \\ \hline
$K_{(1,1)}\Ucal^{\vee} \otimes K_{(2^3)}\Qcal^{\vee}$ & $(10,5)$     \\ \hline
$K_{(1,1)}\Ucal^{\vee} \otimes K_{(1^2)}\Qcal^{\vee}$ & -            \\ \hline
$K_{(2)}\Ucal^{\vee} \otimes K_{(2^3)}\Qcal^{\vee}$   & $(11,6)$, $(11,7)$ \\ \hline
$K_{(2)}\Ucal^{\vee} \otimes K_{(1^2)}\Qcal^{\vee}$   & $(10,6)$     \\ \hline
$K_{(2,1)}\Ucal^{\vee} \otimes K_{(2^3)}\Qcal^{\vee}$ & $(11,7)$     \\ \hline
$K_{(2,1)}\Ucal^{\vee} \otimes K_{(1^2)}\Qcal^{\vee}$ & -            \\ \hline
$K_{(2,2)}\Ucal^{\vee} \otimes K_{(2^3)}\Qcal^{\vee}$ & $(10,6)$     \\ \hline
$K_{(2,2)}\Ucal^{\vee} \otimes K_{(1^2)}\Qcal^{\vee}$ & $(2,2)$      \\ \hline
$K_{(3)}\Ucal^{\vee} \otimes K_{(2^3)}\Qcal^{\vee}$   & $(6,3)$, $(11,7)$, $(11,8)$ \\ \hline
$K_{(3)}\Ucal^{\vee} \otimes K_{(1^2)}\Qcal^{\vee}$   & $(10,7)$     \\ \hline
$K_{(3,1)}\Ucal^{\vee} \otimes K_{(2^3)}\Qcal^{\vee}$ & $(11,7)$, $(11,8)$ \\ \hline
$K_{(3,1)}\Ucal^{\vee} \otimes K_{(1^2)}\Qcal^{\vee}$ & $(10,7)$     \\ \hline
$K_{(3,2)}\Ucal^{\vee} \otimes K_{(2^3)}\Qcal^{\vee}$ & $(11,8)$     \\ \hline
$K_{(3,2)}\Ucal^{\vee} \otimes K_{(1^2)}\Qcal^{\vee}$ & $(2,2)$      \\ \hline
$K_{(3,3)}\Ucal^{\vee} \otimes K_{(2^3)}\Qcal^{\vee}$ & $(3,2)$, $(10,7)$ \\ \hline
$K_{(3,3)}\Ucal^{\vee} \otimes K_{(1^2)}\Qcal^{\vee}$ & $(2,2)$      \\ \hline
\end{tabular}
\end{table}

\newpage

\vspace{8pt} \noindent \hypertarget{Fano1-11}{\textbf{Fano 1-11.}} Hypersurface of degree 6 in $\PP(1^3,2,3)$. \\
\noindent \textit{Description:} $\Lfr(\Ocal(6)) \subset \PP(1^3,2,3)$.

We have the action of
\[
    G = (\Hom(\Sym^2 \CC^3, \CC) \times \Hom(\Sym^3 \CC^3, \CC)) \rtimes \GL_3
\]
on $A = \PP(1^3,2,3)$ as in Section~\ref{subsec:weighted-projective-spaces}.
Write $A = \Proj \CC[s_1, s_2, s_3, t_1, t_2]$ with $\deg s_1 = \deg s_2 = \deg s_3 = 1$, $\deg t_1 = 2$ and $\deg t_3 = 3$.
This $G$-action has a dense orbit $A_0 = A \setminus \Zcal(s_1, s_2, s_3)$.
By Theorem~\ref{thm:Almost-section-dominating-on-WPS}, $\Ocal_{A}(6)$ is section-dominating for $\Ocal_{A}(a)$ on $A_0$ provided $a \geq 8$.
As in \hyperlink{Fano1-1}{\textbf{1-1}}, one can see that $S$ is algebraically hyperbolic when $a \geq 9$.

For $4 \leq a \leq 8$, consider the diagram
\[
\begin{tikzcd}[column sep=small]
    & & & 0 \ar[d] \\
    & & & M_{\Ocal_A(a')}(a'') \ar[d] \\
    0 \ar[r] & H^0(\Ocal_A(a')) \otimes M_{\Ocal_A(a'')} \ar[r] \ar[d] & H^0(\Ocal_A(a')) \otimes H^0(\Ocal_A(a'')) \otimes \Ocal_A \ar[r] \ar[d] & H^0(\Ocal_A(a')) \otimes \Ocal_A(a'') \ar[d] \\
    0 \ar[r] & M_{\Ocal_A(a)} \ar[r] &  H^0(\Ocal_A(a)) \otimes \Ocal_A \ar[r] \ar[d] & \Ocal_A(a) \\
    & & K \otimes \Ocal_A &
\end{tikzcd}
\]
where $a' = \lfloor \frac{a}{2} \rfloor$, $a'' = a-a'$ and $K$ is the cokernel of the multiplication map
\[
\begin{tikzcd}
    H^0(\Ocal_A(a')) \otimes H^0(\Ocal_A(a'')) \ar[r] & H^0(\Ocal_A(a)).
\end{tikzcd}
\]
Note that $\Ocal_S(1)$ is globally generated on $S$ as $S$ avoids $A \setminus A_0$.
Also for $k \leq \min\{ 5,a-1 \}$, the restriction map gives an isomorphism $H^0(\Ocal_A(k)) \simeq H^0(\Ocal_S(k))$, so $M_{\Ocal_S(k)} = M_{\Ocal_A(k)}|_S$.
Thus $\Ocal_S(a')$, $\Ocal_S(a'')$ and $M_{\Ocal_S(a')}(a'') = M_{\Ocal_S(a')}(a') \otimes \Ocal_S(a''-a')$ are all globally generated on $S$, we have an exact sequence
\[
\begin{tikzcd}
    M_{\Ocal_S(a')}(a'')|_C \ar[r] & \Kcal|_C \ar[r] & K \otimes \Ocal_C \ar[r] & 0
\end{tikzcd}
\]
where $\Kcal$ is the cokernel of the map
\[
\begin{tikzcd}
    H^0(\Ocal_S(a')) \otimes M_{\Ocal_A(a'')} \ar[r] & M_{\Ocal_A(a)}.
\end{tikzcd}
\]

Assume that there is a generic surjection $M_{\Ocal_A(a)}|_C \rightarrow N_{f/S}$.
If the induced map $H^0(\Ocal_A(a')) \otimes M_{\Ocal_A(a'')}|_C \rightarrow N_{f/S}$ has a torsion image $\Tcal$, then there is a generic surjection $\Kcal|_C \rightarrow N' = N_{f/S}/\Tcal$.
Again, if the induced map $M_{\Ocal_A(a')}(a'')|_C \rightarrow N'$ has a torsion image $\Tcal'$, then there is a generic surjection $K \otimes \Ocal_C \rightarrow N'/\Tcal'$, yielding $\deg_C N_{f/S} \geq 0$.
Then we have
\[
    2g-2 \geq (C.K_S) = (a-2) \deg_C \Ocal_A(1)
\]
which is the desired estimate.
On the other hand, if $M_{\Ocal_A(a')}(a'')|_C \rightarrow N'$ is generically surjective, since $M_{\Ocal_A(a')}(a'')|_S = M_{\Ocal_S(a')}(a'')$ is globally generated on $S$, we would get the above estimate as well.
Hence we may assume that there is a generic surjection $M_{\Ocal_A(a'')}|_C \rightarrow N_{f/S}$.
Repeating this procedure, we would arrive at a generic surjection $M_{\Ocal_A(1)}|_C \rightarrow N_{f/S}$.
Since $M_{\Ocal_A(1)}(1)|_S = M_{\Ocal_S(1)}(1)$ is globally generated on $S$, we get $\deg_C N(1) \geq 0$, so
\[
    2g-2 \geq (C.K_S) - \deg_C \Ocal_A(1) = (a-3) \deg_C \Ocal_A(1)
\]
which is again the desired estimate for $a \geq 4$.

In conclusion, $S$ is algebraically hyperbolic if $a \geq 4$. The remaining case is $a = 3$.
Note that for $a = 3$, the surface $S$ can be regarded as a very general surface of degree $6$ in $\PP(1^3, 2)$, which is also an unknown case in \cite[Lemma~3.9, Proposition~3.10]{CR2023}.

\vspace{8pt} \noindent \hypertarget{Fano1-12}{\textbf{Fano 1-12.}} Double cover of $\PP^3$ branched over a quartic surface. \\
\noindent \textit{Description:} $\Lfr(\Ocal(4)) \subset \PP(1^4, 2)$.

Since $\Ocal_{A}(2)$ is section-dominating for $\Ocal_{A}(a)$ on $A_0 = A \setminus \{ [0^4:1] \}$ for $a \geq 2$, one can apply the same argument as in \hyperlink{Fano1-1}{\textbf{1-1}} to prove that $S$ is algebraically hyperbolic for $a \geq 4$.
The remaining case is $a = 3$.

\vspace{8pt} \noindent \hypertarget{Fano1-13}{\textbf{Fano 1-13.}} Cubic threefold. \\
\noindent \textit{Description:} $\Lfr(\Ocal(3)) \subset \PP^4$.

As in \hyperlink{Fano1-2}{\textbf{1-2}}, one can see that $S$ is algebraically hyperbolic if and only if $a \geq 3$.

\vspace{8pt} \noindent \hypertarget{Fano1-14}{\textbf{Fano 1-14.}} Complete intersection of two quadrics in $\PP^5$. \\
\noindent \textit{Description:} $\Lfr(\Ocal(2)^{\oplus 2}) \subset \PP^5$.

The same as in \hyperlink{Fano1-13}{\textbf{1-13}}.

\vspace{8pt} \noindent \hypertarget{Fano1-15}{\textbf{Fano 1-15.}} Section of Pl\"ucker embedding of $\Gr(2,5)$ by codimension 3 subspace. \\
\noindent \textit{Description:} $\Lfr(\Ocal(1)^{\oplus 3}) \subset \Gr(2,5)$.

The same as in \hyperlink{Fano1-13}{\textbf{1-13}}.

\vspace{8pt} \noindent \hypertarget{Fano1-16}{\textbf{Fano 1-16.}} Quadric threefold. \\
\noindent \textit{Description:} $\Lfr(\Ocal(2)) \subset \PP^4$.

As in \hyperlink{Fano1-2}{\textbf{1-2}}, one can see that $S$ is algebraically hyperbolic if and only if $a \geq 4$.

\vspace{8pt} \noindent \hypertarget{Fano1-17}{\textbf{Fano 1-17.}} Projective space $\PP^3$. \\
\noindent \textit{Description:} $\PP^3$.

This is the known case by \cite{CR2019}: $S$ is algebraically hyperbolic if and only if $a \geq 5$.

\begin{rmk}
    Using the argument of \hyperlink{Fano1-11}{\textbf{1-11}}, one can show that surfaces of the optimal degree in \hyperlink{Fano1-1}{\textbf{1-1}}, \hyperlink{Fano1-11}{\textbf{1-11}} or \hyperlink{Fano1-12}{\textbf{1-12}} do not contain rational curves.
    Further, one obtains that a very general surface of degree $9$ in $\PP(1,1,1,3)$ is algebraically hyperbolic since $\Ocal(1)$ is globally generated on $\PP(1,1,1,3) \setminus \{ [0:0:0:1] \}$.
    This resolves one unknown case in \cite[Lemma~3.9 and Proposition~3.10]{CR2023}.
\end{rmk}

%======= Bibliography
\bibliographystyle{amsplain}

\begin{thebibliography}{10}

\bibitem{Brody1978} R. Brody, \textit{Compact manifolds and hyperbolicity}, Trans. Amer. Math. Soc. \textbf{235} (1978), 213–219.

\bibitem{CGL2022} X. Chen, F. Gounelas, and C. Liedtke, \textit{Curves on K3 surfaces}, Duke Math. J. \textbf{171} (2022), no. 16, 3283–3362.

\bibitem{CR2004} H. Clemens and Z. Ran, \textit{Twisted genus bounds for subvarieties of generic hypersurfaces}, Amer. J. Math. \textbf{126} (2004), no. 1, 89-120.

\bibitem{CCGK2016} T. Coates, A. Corti, S. Galkin, and A. Kasprzyk, \textit{Quantum periods for 3-dimensional Fano manifolds}, Geom. Topol. \textbf{20} (2016), no. 1, 103–256. 

\bibitem{CR2019} I. Coskun and E. Riedl, \textit{Algebraic hyperbolicity of the very general quintic surface in $\PP^3$}, Adv. Math. \textbf{350} (2019), 1314–1323.

\bibitem{CR2023} I. Coskun and E. Riedl, \textit{Algebraic hyperbolicity of very general surfaces}, Israel J. Math. \textbf{253} (2023), no. 2, 787–811.

\bibitem{DBFT2021} L. De Biase, E. Fatighenti, and F. Tanturri, \textit{Fano 3-folds from homogeneous vector bundles over Grassmannians}, Rev. Mat. Complut. \textbf{35} (2022), no. 3, 649–710. 

\bibitem{Demailly1997} J.-P. Demailly, \textit{Algebraic criteria for Kobayashi hyperbolic projective varieties and jet differentials}, Proc. Symp. Pure Math. \textbf{62} (1997), 285–360.

\bibitem{Dolgachev1982} I. Dolgachev, \textit{Weighted projective varieties}, Group Actions and Vector Fields (J. B. Carrell, ed.), Lecture Notes in Math., vol. 956, Springer-Verlag, Berlin, 1982, pp. 34–71.

\bibitem{Ein1988} L. Ein, \textit{Subvarieties of generic complete intersections}, Invent. Math. \textbf{94} (1988), no. 1, 163–169.

\bibitem{Ein1991} L. Ein, \textit{Subvarieties of generic complete intersections. II}, Math. Ann. \textbf{289} (1991), no. 3, 465–471.

\bibitem{Fulton1984} W. Fulton, \textit{Intersection theory}, Ergeb. Math. Grenzgeb. (3) [Results in Mathematics and Related Areas (3)], vol. 2, Springer-Verlag, Berlin, 1984.

\bibitem{HI2021} C. Haase and N. Ilten, \textit{Algebraic hyperbolicity for surfaces in toric threefolds}, J. Algebraic Geom. \textbf{30} (2021), no. 3, 573–602.

\bibitem{Hartshorne1977} R. Hartshorne, \textit{Algebraic geometry}, Grad. Texts in Math., vol. 52, Springer-Verlag, New York-Heidelberg, 1977.

\bibitem{HK2015} J.-M. Hwang and H. Kim, \textit{Varieties of minimal rational tangents on Veronese double cones}, Algebr. Geom. \textbf{2} (2015), no. 2, 176–192. 

\bibitem{Iskovskikh1978} V. A. Iskovskikh, \textit{Fano threefolds. II}, Izv. Akad. Nauk SSSR Ser. Mat. \textbf{42} (1978), no. 3, 504–549.

\bibitem{Mioranci2023} L. Mioranci, \textit{Algebraic hyperbolicity of very general hypersurfaces in homogeneous varieties}, arXiv:2307.10461, 2023.

\bibitem{MY2024} J. Moraga and W. Yeong, \textit{A hyperbolicity conjecture for adjoint bundles}, arXiv:2412.01811, 2024.

\bibitem{Mukai1989} S. Mukai, \textit{Biregular classication of Fano 3-folds and Fano manifolds of coindex 3}, Proc. Nat. Acad. Sci. U.S.A. \textbf{86} (1989), 3000–3002.

\bibitem{Pacienza2003} G. Pacienza, \textit{Rational curves on general projective hypersurfaces}, J. Algebraic Geom. \textbf{12} (2003), no. 2, 245–267.

\bibitem{RS2009} G. V. Ravindra and V. Srinivas, \textit{The Noether-Lefschetz theorem for the divisor class group}, J. Algebra \textbf{322} (2009), no. 9, 3373–3391.

\bibitem{Robins2023} S. Robins, \textit{Algebraic hyperbolicity for surfaces in smooth projective toric threefolds with Picard rank 2 and 3}, Beitr. Algebra Geom. \textbf{64} (2023), no. 1, 1–27.

\bibitem{Voisin1996} C. Voisin, \textit{On a conjecture of Clemens on rational curves on hypersurfaces}, J. Differential Geom. \textbf{44} (1996), no. 1, 200–213.

\bibitem{Voisin1998} C. Voisin, \textit{A correction: ``On a conjecture of Clemens on rational curves on hypersurfaces''}, J. Differential Geom. \textbf{49} (1998), no. 3, 601–611.

\bibitem{Weyman2003} J. Weyman, \textit{Cohomology of vector bundles and syzygies}, Cambridge Tracts in Mathematics, vol. 149, Cambridge University Press, Cambridge, 2003.

\bibitem{Yeong2022} W. Yeong, \textit{Algebraic hyperbolicity of very general hypersurfaces in products of projective spaces}, (accepted to Israel J. Math), arXiv:2203.01392, 2022.

\end{thebibliography}

\end{document}